\documentclass[a4paper,12pt]{paper}

\usepackage{graphicx}
\usepackage{amsmath,amsthm}
\usepackage{amsfonts,amssymb}
\usepackage{amscd}
\usepackage{color}
\usepackage[cp1251]{inputenc}
\usepackage[matrix,arrow,curve]{xy}
\usepackage{tikz}
\usetikzlibrary{arrows}
\usepackage{mathtools}

\newtheorem{lemma}{Lemma}[section]

\newtheorem{theorem}[lemma]{Theorem}
\newtheorem{coro}[lemma]{Corollary}
\newtheorem{rema}[lemma]{{\rm\bf Remark}}
\newtheorem{Def}[lemma]{Definition}

\newtheorem{ex}[lemma]{Example}

\newcommand{\charr}{{\rm char\,}}

\newcommand{\kk}{{\bf k}}

\newcommand{\Ann}{{\rm Ann}}

\renewcommand{\Im}{{\rm Im\,}}

\newcommand{\la}{\langle}
\newcommand{\ra}{\rangle}

\renewcommand{\le}{\leqslant}
\renewcommand{\ge}{\geqslant}

\numberwithin{equation}{section}
\begin{document}

\title{Anticommutative Engel algebras of the first five levels.}

\author{Yury Volkov}
\maketitle
\begin{abstract}
Anticommutative Engel algebras of the first five degeneration levels are classified. All algebras appearing in this classification are nilpotent Malcev algebras.
\end{abstract}


       \vspace{0.3cm}


       \vspace{0.3cm}

\section{Introduction}

Algebras in this paper are not assumed to be associative. The main object considered in this paper is the  degeneration of algebras.
Roughly speaking, the algebra $A$ degenerates to the algebra $B$ if there is a family of algebra structures parameterized by an element of the ground field such that infinitely many structures in the family represent $A$ and there exists a structure belonging to this family representing $B$. Note that the notion of a degeneration is closely related to the notions of contraction and deformation.

The notion of the level of an algebra was introduced in \cite{gorb91}. The algebra under consideration is an algebra of level $n$ if the maximal length of a chain of non-trivial degenerations starting at it equals $n$.
Roughly speaking, the level estimates the complexity of the multiplication of  the given algebra. For example, the unique algebra of the level zero is the algebra with zero multiplication. 

Anticommutative algebras of the first level were classified correctly in \cite{gorb91} and all algebras of the first level were classified in \cite{khud13} (see also \cite{IvaPal}). In \cite{gorb93} the author introduced the notion of the infinite level. The infinite level can be expressed in terms of the usual level, and because of this  the classification of algebras with a given infinite level is much easier than the classification of algebras with a given usual level. Anticommutative algebras of the second infinite level were classified in \cite{gorb93}.
The classification of algebras of the third infinite level given in the same paper occurs to be incorrect and can not be taken in account. Finally, associative, Lie, Jordan, Leibniz and nilpotent algebras of the level two were classified in \cite{khud15,khud17} before the full classification of algebras of the second level appeared in \cite{kpv19}.

This paper is a natural continuation of \cite{kpv19} and constitute the first natural step in the classification of algebras of the third level and of anticommutative algebras (or, more generally, algebras of the generation type $1$) of the first five levels.
Let us explain why this step is really constitute a part of these classifications. First of all, as explained in \cite{kpv19}, it is natural to classify algebras depending on their generation type, i.e. the maximal dimension of a one-generated subalgebra.
Two main cases in the classification of algebras of the third level are the algebras of generation type $1$ and  the algebras of generation type $2$. As it was shown in \cite{kpv19},  there is also the case of generation type $3$, but there are almost no algebras of level $3$ with generation type $3$ and we will leave this small part of classification for the paper where we will finish the classification of algebras of level $3$. The case of generation type $2$ requires some tedious calculations using the results of \cite{kpv17} and will be done in some of proceeding papers.
The algebras of generation type $1$ admit so-called one-dimensional standard In\"on\"u-Wigner contractions with respect to any element. These contractions are classified for algebras of generation type $1$ until level $5$ in \cite{kpv19}. In the same paper it is explained how obtain their classification until any level. It is natural for our aim to divide the algebras of generation type $1$ to classes depending on what one-dimensional standard In\"on\"u-Wigner contractions they admit. If for an  algebra $A$ of generation type $1$ a one-dimensional standard In\"on\"u-Wigner contraction of maximal possible level is nilpotent, then $A$ is anticommutative and Engel. Since the classification of  one-dimensional standard In\"on\"u-Wigner contractions presented in \cite{kpv19} is divided into nilpotent, solvable and non-solvable cases, it is natural to consider anticommutative Engel algebras first. This is exactly what we will do in this paper. Namely, we will classify these algebras until fifth level.
We will also present the classification of anticommutative Engel algebras of the first five infinite levels that will easily follow from our classification and will not differ from it very much. Note that the class of anticommutative Engel algebras includes the class of anticommutative nilpotent algebras. In this paper we will show that until fifth level these classes coincide.

Note that except the inclusion of one algebra of level four, the classification of anticommutative nilpotent algebras of the third infinite level that can be extracted from \cite{gorb93} is correct and coincides with the classification of anticommutative nilpotent algebras of the third level that we will obtain in this work. Thus, part of our results confirms the nilpotent part of the results of \cite{gorb93}. Contrariwise, the non-nilpotent part of the classification in \cite{gorb93} has more problems and will be corrected in our proceeding paper.

\section{Background on degenerations}

In this section we introduce some notation and recall some well known definitions and results about degenerations that we will need in this work.

All vector spaces in this paper are over some fixed algebraically closed field ${\bf k}$ and we write simply $dim$, $Hom$ and $\otimes$ instead of $dim_{{\bf k}}$, $Hom_{{\bf k}}$ and $\otimes_{{\bf k}}$.
An algebra in this paper is simply a vector space with a bilinear binary operation called multiplication. This operation does not have to be associative unlike to the case of usual algebras. For an algebra $A$ and $a,b\in A$ we will denote the result of the application of multiplication to the pair $(a,b)$ by $ab$. We will write also $a^2$ instead of $aa$. If $V$ is a linear space and $S$ is a subset of $V$, then we denote by $\la S\ra$ the subspace of $V$ generated by $S$. For two subspace $A_1,A_2$ of $A$ we set
$A_1A_2:=\la\{a_1a_2\}_{a_1\in A_1,a_2\in A_2}\ra$.

Let $V$ be a fixed $n$-dimensional space. Then the {\it set of $n$-dimensional algebra structures} on $V$ is $\mathcal{A}_n=Hom(V\otimes V, V)\cong V^*\otimes V^*\otimes V$. Any $n$-dimensional algebra can be represented by some element of $\mathcal{A}_n$. Two algebras are isomorphic if and only if they can be represented by the same structure.
The set $\mathcal{A}_n$ has a structure of the affine variety ${\bf k}^{n^3}$. There is a natural action of the group $GL(V)$ on $\mathcal{A}_n$ defined by the equality $ (g * \mu )(x\otimes y) = g\mu(g^{-1}x\otimes g^{-1}y)$ for $x,y\in V$, $\mu\in \mathcal{A}_n$ and $g\in GL(V)$.
Two structures represent the same algebra if and only if they belong to the same orbit. By $\kk^n$ we will denote the $n$-dimensional algebra with zero multiplication and the structure representing it.
For brevity, we will write $\mu(u,v)$ or, if the structure $\mu$ is clear from the context, even $uv$ instead of $\mu(u\otimes v)$ for $u,v\in V$.

Let $A$ and $B$ be $n$-dimensional algebras. Suppose that $\mu,\chi\in \mathcal{A}_n$ represent $A$ and $B$ respectively. We say that $A$ {\it degenerates} to $B$ and write $A\to B$ if $\chi$ belongs to $\overline{O(\mu)}$. Here, as usually, $O(X)$ denotes the orbit of $X$ and $\overline{X}$ denotes the closure of $X$. We also write $A\not\to B$ if $\chi\not\in\overline{O(\mu)}$. We say that the degeneration $A\to B$ is {\it trivial} if $A\not\cong B$. We will write $A\xrightarrow{\not\cong} B$ to emphasize that the degeneration $A\to B$ is not trivial.

Whenever an $n$-dimensional space named $V$ appears in this paper, we assume that there is some fixed basis $e_1,\dots, e_n$ of $V$. In this case, for $\mu\in\mathcal{A}_n$, we denote by $\mu_{i,j}^k$ ($1\le i,j,k\le n$) the structure constants of $\mu$ in this fixed basis, i.e. elements of $\bf k$ such that $\mu(e_i,e_j)=\sum\limits_{k=1}^n\mu_{i,j}^ke_k$. To prove degenerations and nondegenerations we will use the same technique that has been already used in \cite{S90} and \cite{kppv, kpv17}. In particular, we will be free to use \cite[Lemma 1]{kppv} and facts that easily follow from it.
This lemma asserts the following. If $A\to B$, $\mu\in\mathcal{A}_n$ represents $A$ and there is a closed subset $\mathcal{R}\subset\mathcal{A}_n$ invariant under lower triangular transformations of the basis $e_1,\dots,e_n$ such that $\mu\in\mathcal{R}$, then there is a structure $\chi\in\mathcal{R}$ representing $B$. Invariance under lower triangular transformations of the basis $e_1,\dots,e_n$ means that if $\omega\in\mathcal{R}$ and $g\in GL(V)$ has a lower triangular matrix in the basis $e_1,\dots,e_n$, then $g*\omega\in\mathcal{R}$ (see \cite{kppv} for a more detailed discussion). The mentioned lemma implies, in particular, that if $A\to B$, then $dim\,A^2\ge dim\,B^2$. We will denote by $\Ann(A)$ the set of such $a\in A$ that $aA=Aa=0$. Another consequence of the mentioned lemma states that if $A\to B$, then $dim\,\Ann(A)\le dim\,\Ann(B)$. More generally, let $\lambda\in\mathcal{A}_n$ be an $n$-dimensional algebra structure. For two subspaces $U,W$ of $V$ we will write $\lambda(U,W)$ for the subspace of $V$ generated by $\lambda(u,w)$ for all $u\in U$ and $w\in W$. We also set $V_i=\langle e_i,\dots,e_n\rangle$ for $1\le i\le n+1$. Then a condition of the form $\lambda(V_i,V_j)\subset V_k$ determines a closed subset of $\mathcal{A}_n$ invariant under lower triangular transformations of the basis $e_1,\dots,e_n$. In particular, the condition $dim\,A^2\le m$ is equivalent to the fact that $A$ can be represented by a structure from the set $\{\lambda\in\mathcal{A}_n\mid \lambda(V,V)\subset V_{n-m+1}\}$
and the condition $dim\,\Ann(A)\ge m$ is equivalent to the fact that $A$ can be represented by a structure from the set $\{\lambda\in\mathcal{A}_n\mid \lambda(V,V_{n-m+1})+\lambda(V_{n-m+1},V)=0\}$. If there are integer $s$ and $1\le i_1,\dots,i_s,j_1,\dots,j_s,k_1,\dots,k_s\le n$ such that
$$\mathcal{R}=\{\lambda\in\mathcal{A}_n\mid \lambda(V_{i_1},V_{j_1})\subset V_{k_1},\dots,\lambda(V_{i_s},V_{j_s})\subset V_{k_s}\}$$
satisfies the conditions $O(\mu)\cap \mathcal{R}\not=\varnothing$ and $O(\chi)\cap \mathcal{R}=\varnothing$, where $\mu$ represents $A$ and $\chi$ represents $B$, then we will write $A\not\to_{(i_1,j_1,k_1),\dots,(i_s,j_1,k_1)}B$ to emphasize a reason for the corresponding non-degeneration. Note that $dim\,A^2=m<dim\,B^2$ is equivalent to $A\not\to_{(1,1,n-m+1)}B$ and $dim\,\Ann(A)=m>dim\Ann(B)$ is equivalent to $A\not\to_{(n-m+1,1,n+1),(1,n-m+1,n+1)}B$. In some more complicated situation we will define $\mathcal{R}$ explicitly.

In fact, in this paper we will mainly consider the closed subvariety $\mathcal{AC}_n$ of the variety $\mathcal{A}_n$ formed by anticommutative algebra structures, i.e. structures $\mu$ such that $\mu_{i,i}^k=0$ and $\mu_{i,j}^k+\mu_{j,i}^k=0$ for all $1\le i,j,k,\le n$. In this case we will describe $\mathcal{R}$ by an expression of the form $\mathcal{R}=\{\lambda\in\mathcal{AC}_n\mid \dots\}$. Note that many things simplify in the anticommutative case. For example, $dim\,\Ann(A)=m>dim\Ann(B)$ is equivalent to $A\not\to_{(1,n-m+1,n+1)}B$ for anticommutative algebras $A$ and $B$.

To prove degenerations, we will use the technique of contractions. Namely, let $\mu,\chi\in \mathcal{A}_n$ represent $A$ and $B$ respectively. Suppose that there are some elements $E_i^t\in V$ ($1\le i\le n$, $t\in{\bf k}^*$) such that $E^t=(E_1^t,\dots,E_n^t)$ is a basis of $V$ for any $t\in{\bf k}^*$ and the structure constants of $\mu$ in this basis are $\mu_{i,j}^k(t)$ for some polynomials $\mu_{i,j}^k(t)\in{\bf k}[t]$. If $\mu_{i,j}^k(0)=\chi_{i,j}^k$ for all $1\le i,j,k\le n$, then $A\to B$. To emphasize that the {\it parameterized basis} $E^t=(E_1^t,\dots,E_n^t)$ ($t\in{\bf k}^*$) gives a degeneration between algebras represented by the structures $\mu$ and $\chi$, we will write $\mu\xrightarrow{E^t}\chi$.
Usually we will simply write down the parameterized basis explicitly above the arrow.

An important role in this paper will be played by a particular case of a degeneration called a {\it standard In\"on\"u-Wigner contraction} (see \cite{IW}). We will call it {\it IW contraction} for short.
Suppose that $A_0$ is an $m$-dimensional subalgebra of the $n$-dimensional algebra $A$ and $\mu\in\mathcal{A}_n$ is a structure representing $A$ such that $A_0$ corresponds to the subspace $\langle e_1,\dots,e_m\rangle$ of $V$.
Then $\mu\xrightarrow{(e_1,\dots,e_m,te_{m+1},\dots,te_n)}\chi$ for some $\chi\in\mathcal{A}_n$ and the algebra $B$ represented by $\chi$ is called the IW contraction of $A$ with respect to $A_0$. The isomorphism class of the resulting algebra does not depend on the choice of the structure $\mu$ satisfying the condition stated above and always has an ideal $I\subset B$ and a subalgebra $B_0\subset B$ such that $B=B_0\oplus I$ as a vector space, $I^2=0$ and $B_0\cong A_0$ as an algebra. We will call an algebra of such a form a {\it trivial singular extension} of $A_0$ by ${\bf k}^{n-m}$.

To finish this section, let us introduce the notion of a level related to the notion of a degeneration. This notion will be the main object of interest in this paper.

\begin{Def}{\rm
The {\it level} of the $n$-dimensional algebra $A$ is the maximal number $m$ such that there exists a sequence of non-trivial degenerations $A\xrightarrow{\not\cong}A_{m-1}\xrightarrow{\not\cong}\dots\xrightarrow{\not\cong}A_1\xrightarrow{\not\cong}A_0$ for some $n$-dimensional algebras $A_i$ ($0\le i\le m-1$). The level of $A$ is denoted by $lev(A)$. The {\it infinite level} of the algebra $A$ is the number defined by the equality $lev_{\infty}(A)=\lim\limits_{m\to\infty}lev(A\oplus\kk^m)$.
}\end{Def}

The aim of this paper is to classify up to isomorphism the anticommutative Engel algebras with level not greater than $5$. This will automatically give us also the classification of algebras with infinite level not greater than $5$ in the same variety.

\section{Generation type one and one-dimensional IW contractions}

In this section we recall some general ideas of \cite{kpv19} on how to classify algebras of small levels. Let us first recall the definition of a generation type.

\begin{Def}{\rm
Let $A$ be an $n$-dimensional algebra. For $a\in A$, we denote by $A(a)$ the subalgebra of $A$ generated by $a$. The {\it generation type} of $A$ is the dimension of a maximal $1$-generated subalgebra of $A$, i.e. the number $G(A)$ defined by the equality $G(A)=\max\limits_{a\in A}\big(dim\,A(a)\big)$.
}\end{Def}

By the results of \cite{kpv19}, if $G(A)\ge 3$ for an $n$-dimensional algebra $A$, then $lev(A)\ge G(A)$. Moreover, there are no many algebras with $G(A)=3$ that can have level $3$ and all of them are described in the same work. Thus, the main problems in the classification of algebras of level $3$ are the classifications of algebras of level $3$ with generation types one and two. Moreover, a more detailed consideration would show that this cases constitute the main parts of classifications of algebras of levels not greater than $5$. The case of generation type $2$ will be considered in one of our proceeding works and at this moment it seems to be difficult to classify algebras with generation type $2$ that have levels four and five. Nevertheless, we are going to classify algebras with generation type one that have levels not greater than $5$. The first part of this classification we present in this paper.

\begin{Def}{\rm
The algebra $A$ is called {\it anticommutative} if $a^2=0$ for any $a\in A$.  The algebra $A$ is called {\it nilpotent} if there exists $m$ such that $A^m=0$, where we define $A^i$ by induction on $i\ge 1$ in the following way. We set $A^1=1$ and $A^i=A(A^{i-1})+(A^{i-1})A$ for $i>1$. The algebra $A$ is called {\it $m$-Engel} if $(L_a)^m=0$ for any $a\in A$. We will call the algebra $A$ Engel if it is $m$-Engel for some $m>0$.
}\end{Def}

Let $A$ be an $n$-dimensional algebra. If $G(A)=1$, then for any $a\in A$ the IW contraction of $A$ with respect to $A(a)$ is defined. We will denote the resulting algebra by $IW_a(A)$. Algebras of the form $IW_a(A)$ with $G(A)=1$ and $a\in A$ were studied in \cite{kpv19}. Their degenerations are well understood due to the results of the last mentioned paper. It is natural to consider separately the case where $IW_a(A)$ is nilpotent for any $a\in A$ and the case where there exists $a\in A$ such that $IW_a(A)$ is not nilpotent. Since in the first case the algebra $A$ clearly does not have idempotents, it is anticommutative. During this paper, for $a\in A$, we will denote by $L_a$ the operator of left multiplication by $a$, i.e. $L_a$ is a linear map from $A$ to itself defined by the equality $L_a(b)=ab$ for $b\in A$.
For anticommutative $A$, the nilpotence of $IW_a(A)$ is equivalent to the nilpotence of the operator induced by $L_a$ on the space $A/\la a\ra$. Note that $dim\,A/\la a\ra=n-1$. Hence, if $IW_a(A)$ is nilpotent, then $L_a^k(A)\subset \la a\ra$ for some integer $0<k<n$. On the other hand, if $L_a^k(b)=\alpha a$ for some $b\in A$ and $\alpha\in\kk^*$, then $IW_{L_a^{k-1}(b)}(A)$ is not nilpotent that contradicts our assumptions.
Thus, $A$ is Engel and the minimal integer $m$ such that $(L_a)^m=0$ for all $a\in A$ is the same as the minimal integer such that $IW_a(A)^{m+1}=0$ for all $a\in A$.
Thus, the consideration of algebras with generation type one that have only nilpotent one-dimensional IW contractions is equivalent to the consideration of anticommutative Engel algebras. This motivated us to classify first anticommutative Engel algebras until  the fifth level.

\begin{rema}\label{nileng} It is clear that any nilpotent algebra is Engel. It follows also from \cite[Theorem 4]{Kuz} that any finite dimensional anticommutative $3$-Engel algebra is nilpotent.
On the other hand, due to the examples of \cite{KH} finite dimensional anticommutative $4$-Engel algebra does not have to be nilpotent.
\end{rema}

\begin{lemma}\label{maxcont} Let $A$ be an $n$-dimensional anticommutative Engel algebra. There exists $c\in A$ such that $IW_c(A)\to IW_a(A)$ for any $a\in A$.
\end{lemma}
\begin{proof} For $a\in A$, we denote by $r_m(a)$ the rank of the operator $(L_a)^m$. Due to the results of \cite{kpv19}, $IW_a(A)\to IW_b(A)$ if and only if $r_m(a)\ge r_m(b)$ for any $m>0$. Let us pick $c$ such that $IW_a(A)\not\to IW_c(A)$ whenever  $IW_a(A)\not\cong IW_c(A)$ for some $a\in A$. Suppose that $IW_c(A)\not\to IW_b(A)$ for some $b\in A$. This means that $r_{m_0}(b)>r_{m_0}(c)$ for some $m_0>0$. Let us consider the elements $c+\alpha b$ with $\alpha\in\kk$. Since the condition $r_m(x)\ge max(r_m(b), r_m(c))$ determines an open subset of $A$ considered as a affine variety $\kk^n$ with Zariski topology, for a fixed $m$, we have $r_m(c+\alpha b)\ge max(r_m(b), r_m(c))$ for all $\alpha\in\kk$ except a finite number of values. Since the mentioned inequality is satisfied for $m>n$, there exists $\alpha\in\kk$ such that $r_m(c+\alpha b)\ge max(r_m(b), r_m(c))$ for any $m>0$. Since $r_{m_0}(c+\alpha b)>r_{m_0}(c)$, we have $IW_{c+\alpha b}(A)\not\cong IW_c(A)$. On the other hand, it follows from the argument above that $IW_{c+\alpha b}(A)\to IW_c(A)$ that contradicts the choice of $c$.
\end{proof}

It follows from Lemma \ref{maxcont} that any $n$-dimensional anticommutative Engel algebra $A$ has a unique one-dimensional IW contraction of maximal level. We will denote this contraction by $IW_1^{max}(A)$. We will use in this paper also the next auxiliary fact.

\begin{lemma} Let $A$ and $B$ be $n$-dimensional anticommutative Engel algebras. If $A\to B$, then $IW_1^{max}(A)\to IW_1^{max}(B)$.
\end{lemma}
\begin{proof} Let us denote by $L^B_b:B\rightarrow B$ the  operator of left multiplication by $b\in B$ and  by $L^A_a:A\rightarrow A$ the operator of left multiplication by $a\in A$.
If $IW_1^{max}(A)\not\to IW_1^{max}(B)$, then there is some $b\in B$ and integer $m$ such that the rank $R$ of $(L_b^B)^m$ is greater than  the rank of $(L_a^A)^m$ for any $a\in A$. It is not difficult to see that the set of structures representing algebras $C$ such that the rank of $(L_c^C)^m$ is less than $R$ for any $c\in C$ is a closed subset of $\mathcal{AC}_n$. It is clear that $A$ can be represented by a structure from this subset and $B$ cannot. Thus, $A\not\to B$.
\end{proof}

\section{Nilpotent one-dimensional IW contractions of small levels}

From here on we consider only anticommutative Engel algebras. Any algebra that will appear is assumed to be so if the opposite is not stated.

Our strategy is to classify separately algebras with different $IW_1^{max}(A)$. Note that $lev\big(IW_1^{max}(A)\big)\le lev(A)$, and hence to classify anticommutative Engel algebras of the first five levels, we need the classification of their possible one-dimensional IW contractions until level five. Such a classification is presented in \cite{kpv19} and we give it here with small changes corresponding to permutations of basic elements.

\begin{center}
Table 1. {\it Nilpotent one-dimensional IW contractions of algebras with generation type $1$ of the first $5$ levels.}\\
\begin{tabular}{|l|l|l|l|}
\hline
{\rm level}&{\rm notation}&{\rm multiplication table}&{\rm dimension}\\
\hline
{\rm 1}&$\begin{array}{l}{\bf n}_3\end{array}$&$\begin{array}{l}e_1e_2=e_n\end{array}$&$\begin{array}{l}n\ge 3\end{array}$\\
\hline
\hline
{\rm 2} &$\begin{array}{l}T^3\vspace{0.1cm}\\T^{2,2}\end{array}$ & $\begin{array}{l}e_1e_2=e_3,\,\,e_1e_3=e_4\vspace{0.1cm}\\e_1e_2=e_{n-1},\,\,e_1e_3=e_n\end{array}$&$\begin{array}{l}n=4\vspace{0.1cm}\\n\ge 5\end{array}$\\
\hline
\hline
{\rm 3} &$\begin{array}{l}T^3\vspace{0.1cm}\\T^{2,2,2}\end{array}$ & $\begin{array}{l}e_1e_2=e_3,\,\,e_1e_3=e_n\vspace{0.1cm}\\e_1e_{i+1}=e_{i+n-3},\,\,1\le i\le 3\end{array}$&$\begin{array}{l}n\ge 5\vspace{0.1cm}\\n\ge 7\end{array}$\\
\hline
\hline
{\rm 4}&$\begin{array}{l}T^4\vspace{0.1cm}\\T^{3,2}\vspace{0.1cm}\\T^{2,2,2,2}\end{array}$ & $\begin{array}{l}e_1e_i=e_{i+1},\,\,2\le i\le 4\vspace{0.1cm}\\
e_1e_2=e_{n-1},\,\,e_1e_3=e_4,\,\,e_1e_4=e_n\vspace{0.1cm}\\e_1e_{i+1}=e_{i+n-4},\,\,1\le i\le 4\end{array}$&$\begin{array}{l}n=5\vspace{0.1cm}\\n\ge 6\vspace{0.1cm}\\n\ge 9\end{array}$\\
\hline
\hline
{\rm 5}&$\begin{array}{l}T^4\vspace{0.1cm}\\T^{3,3}\vspace{0.1cm}\\T^{3,2,2}\vspace{0.1cm}\\T^{2,2,2,2,2}\end{array}$ & $\begin{array}{l}e_1e_2=e_3,\,\,e_1e_3=e_4,\,\,e_1e_4=e_n\vspace{0.1cm}\\e_1e_i=e_{i+1},\,\,i\in\{2,3,5,6\}\vspace{0.1cm}\\ e_1e_2=e_{n-2},\,\,e_1e_3=e_{n-1},\,\,e_1e_4=e_5,\,\,e_1e_5=e_n\vspace{0.1cm}\\e_1e_{i+1}=e_{i+n-5},\,\,1\le i\le 5\end{array}$&$\begin{array}{l}n\ge 6\vspace{0.1cm}\\n=7\vspace{0.1cm}\\n\ge 8\vspace{0.1cm}\\n\ge 11\end{array}$\\
\hline
\end{tabular}
\end{center}
Here and in all other multiplication tables, we give only nonzero products of the form $e_ie_j$ with $i<j$. The values of products of basic elements that are not determined by the given ones and the anticommutativity are zero.

It follows from the results of \cite{kpv19} that if $IW_1^{max}(A)$ can be represented by ${\bf n}_3$, then $A$ is isomorphic to one of the Heisenberg Lie algebras defined in the next table.
\begin{center}
Table 2. {\it Heisenberg Lie algebras.}\\
\begin{tabular}{|l|l|l|l|}
\hline
{\rm level}&{\rm notation}&{\rm multiplication table}&{\rm dimension}\\
\hline
{\rm m}&$\begin{array}{l}\eta_m\end{array}$&$\begin{array}{l}e_{2i-1}e_{2i}=e_{2m+1},\,\,1\le i\le m\end{array}$&$\begin{array}{l}n\ge 2m+1\end{array}$\\
\hline
\end{tabular}
\end{center}

This immediately gives the classification of anticommutative Engel algebras of level two.

\begin{theorem}[\cite{kpv19}] Let $A$ be an $n$-dimensional anticommutative Engel algebra of level two. Then either $n=4$ and $A$ can be represented by $T^3$ or $n\ge 5$ and  $A$ can be represented by $T^{2,2}$ or $\eta_2$.
\end{theorem}

Since, for an algebra $A$ of level not greater than $5$ such that $IW_1^{max}(A)$ has level five, one obviously has $A\cong IW_1^{max}(A)$, we need to consider algebras with $IW_1^{max}(A)$ represented by a structure from the set $\{T^{2,2},T^{2,2,2},T^{2,2,2,2},T^3,T^{3,2},T^4\}$ to finish our classification. All of these algebras except $T^4$ are $3$-Engel, and hence nilpotent by Remark \ref{nileng}. Moreover, we need to consider the case of the algebra $T^4$ only in the dimension $5$.

Note that any nilpotent algebra $A$ can be represented by a structure $\mu\in \mathcal{A}_n$ such that $\mu_{i,j}^k=0$ for $k\le max(i,j)$. Note that if $v\in V$ is such that $IW_v(\mu)\cong IW_1^{max}(\mu)$, then $IW_{e_1+\alpha v}(\mu)\cong IW_1^{max}(\mu)$ for all $\alpha\in\kk$ except finite number of values (see the proof of Lemma \ref{maxcont}). Thus, we may assume that $\mu_{i,j}^k=0$ for $k\le max(i,j)$ and $IW_{e_1}(\mu)\cong IW_1^{max}(\mu)$ at the same time. Note that this properties are preserved with respect to lower triangular transformations $g\in GL(V)$ such that $g(e_1)=e_1$. Then we may assume that  $IW_{e_1}(\mu)$ is exactly one of the structures described in Table 1 up to some permutation of the basic elements $e_2,\dots,e_n$.
Later in all cases, except the case $IW_1^{max}(A)\cong T^4$, we will represent $A$ by a structure $\mu$ satisfying the described conditions.

\section{Algebras with maximal IW contraction $T^{2,2}$}

This section is devoted to the classification of algebras $A$ such that $IW_1^{max}(A)\cong T^{2,2}$.

Let us start with a general observation about algebras $A$ such that $IW_1^{max}(A)\cong T^{\overbrace{\scriptstyle 2,\dots,2}^{m}}$, where $m$ is an arbitrary integer.
Such an algebra can be represented by a structure $\mu$ such that $\mu(e_1,e_{i_r})=e_{j_r}$ and $\mu(e_1,e_i)=0$ for $i\not\in\{i_1,\dots,i_m\}$, where  $2\le i_1,\dots,i_m,j_1,\dots,j_m\le n$ are $2m$ different integers such that $i_r<j_r$ for all $1\le r\le m$.
Without loss of generality we will assume that $2\le i_1<\dots<i_m\le n$. We may assume that at the same time $\mu_{i,j}^k=0$ if $k\le max(i,j)$.

\begin{lemma}\label{2k2} In the settings described above
\begin{enumerate}
\item if $\mu_{i,j}^k\not=0$ for some $2\le i,j,k\le n$, then either $k\in\{j_1,\dots,j_m\}$ or $i,j\in\{i_1,\dots,i_m\}$;
\item for $1\le r,s\le m$ and $2\le i\le n$, one has $\mu_{i,j_s}^{j_r}+\mu_{i,i_s}^{i_r}=0$.
\end{enumerate}
\end{lemma}
\begin{proof}\begin{enumerate}
\item Suppose that $k\not\in\{j_1,\dots,j_m\}$ and $j\not\in\{i_1,\dots,i_m\}$. Let us consider the element $v_{\alpha}=e_1+\alpha e_i$ for $\alpha\in \kk$. Note that $L_{v_{\alpha}}(e_{i_r})=e_{j_r}+\alpha\mu(e_i,e_{i_r})$ and $L_{v_{\alpha}}(e_j)=\alpha\mu(e_i,e_j)$. Hence, the matrix of $L_{v_{\alpha}}$ in the basis $e_1,\dots,e_n$ contains the $(m+1)\times (m+1)$ minor
$$
\begin{vmatrix}
1+\alpha\mu_{i,i_1}^{j_1}&\alpha\mu_{i,i_2}^{j_1}&\cdots&\alpha\mu_{i,i_m}^{j_1}&\alpha\mu_{i,j}^{j_1}\\
\alpha\mu_{i,i_1}^{j_2}&1+\alpha\mu_{i,i_2}^{j_2}&\cdots&\alpha\mu_{i,i_m}^{j_2}&\alpha\mu_{i,j}^{j_2}\\
\vdots&\vdots&\ddots&\vdots&\vdots\\
\alpha\mu_{i,i_1}^{j_m}&\alpha\mu_{i,i_2}^{j_m}&\cdots&1+\alpha\mu_{i,i_m}^{j_m}&\alpha\mu_{i,j}^{j_m}\\
\alpha\mu_{i,i_1}^{k}&\alpha\mu_{i,i_2}^{k}&\cdots&\alpha\mu_{i,i_k}^{k}&\alpha\mu_{i,j}^{k}\\
\end{vmatrix}
$$
which is a polynomial in $\alpha$ with coefficient of the term $\alpha$ equal to $\alpha\mu_{i,j}^{k}$. This means that this polynomial is not constantly zero, and hence, for some $\alpha\in\kk$, the rank of $L_{v_{\alpha}}$ is not less than $m+1$ that contradicts $IW_1^{max}(A)\cong T^{\overbrace{\scriptstyle 2,\dots,2}^{m}}$.	

\item By our assumptions, we have $\left(L_{e_1+e_i}\right)^2=\left(L_{e_1}\right)^2=\left(L_{e_i}\right)^2=0$, and hence $L_{e_1}L_{e_i}+L_{e_i}L_{e_1}=0$. On the other hand,
$$(L_{e_1}L_{e_i}+L_{e_i}L_{e_1})(e_{i_s})=\mu\big(e_1,\mu(e_i,e_{i_s})\big)+\mu(e_i,e_{j_s}).$$
Calculating the coefficient of $e_{j_r}$ in the obtained expression, one gets the required equality.
\end{enumerate}
\end{proof}

\begin{coro}\label{2k2sf} If $IW_1^{max}(A)\cong T^{\overbrace{\scriptstyle 2,\dots,2}^{m}}$, then $A$ can be represented by a structure $\mu$ such that $\mu(e_1,e_{i+1})=e_{i+n-m}$ for $1\le i\le m$,  $\mu(e_1,e_i)=0$ for $m+2\le i\le n$ and $\mu_{i,j}^k=0$ if $k\le max(i,j)$.
\end{coro}
\begin{proof} It is enough to take the structure $\mu$ described above and consider it in the basis $e_1,e_{i_1},\dots,e_{i_m},e_{k_1},\dots,e_{k_{n-2m-1}},e_{j_1},\dots,e_{j_m}$, where $k_1,\dots,k_{n-2m-1}$ are all elements of $\{2,\dots,n\}\setminus\{i_1,\dots,i_m,j_1,\dots,j_m\}$ in the increasing order. Lemma \ref{2k2} guarantees that the new structure $\tilde\mu$ still satisfies the condition $\tilde\mu_{i,j}^k=0$ for $k\le max(i,j)$.
\end{proof}

Let us now return to the case $IW_1^{max}(A)\cong T^{2,2}$. Let us introduce the algebra
\begin{center}
\begin{tabular}{|l|l|l|}
\hline
{\rm notation}&{\rm multiplication table}&{\rm dimension}\\
\hline
$\begin{array}{l}T^{2,2}({\scriptstyle\epsilon_{23}^{n-2}})\end{array}$&$\begin{array}{l}e_1e_2=e_{n-1},\,\,e_1e_3=e_n,\,\,e_2e_3=e_{n-2}\end{array}$&$\begin{array}{l}n\ge 6\end{array}$\\
\hline
\end{tabular}
\end{center}

Let $U$ be an $(n-2)$-dimensional vector space and $\phi:U\times U\rightarrow \kk^2$ be a skew-symmetric bilinear map. We define a binary product on the space $U\oplus\kk^2$ by the equality $(u_1,v_1)(u_2,v_2)=\big(0,\phi(u_1,u_2)\big)$ and denote the resulting algebra by $U\ltimes_{\phi}\kk^2$.

\begin{coro}\label{T22class} One has $IW_1^{max}(A)\cong T^{2,2}$ if and only if $A$ either can be represented by $T^{2,2}({\scriptstyle\epsilon_{23}^{n-2}})$ or is isomorphic to $U\ltimes_{\phi}\kk^2$ for some $(n-2)$-dimensional vector space $U$ and some surjective skew-symmetric bilinear map $\phi:U\times U\rightarrow \kk^2$.
\end{coro}
\begin{proof} It is easy to check that $IW_1^{max}\big(T^{2,2}({\scriptstyle\epsilon_{23}^{n-2}})\big)=T^{2,2}$. If $A\cong U\ltimes_{\phi}\kk^2$, then clearly $A^3=0$, $\dim A^2=2$, and hence $IW_1^{max}(A)$ can be represented either by ${\bf n}_3$ or by $T^{2,2}$. But in the first case, one has $A\cong \eta_m$ for some integer $m$ and, in particular, $\dim A^2=1$. The obtained contradiction shows that $IW_1^{max}(A)\cong T^{2,2}$.

Suppose now that $IW_1^{max}(A)\cong T^{2,2}$. Let us represent $A$ by a structure $\mu$ satisfying conditions of Corollary \ref{2k2sf}. Lemma \ref{2k2} implies that $\mu_{i,j}^k=0$ if $k<n-1$, $2\le i,j\le n$ and $\{i,j\}\not=\{2,3\}$. By the same lemma, we have also $\mu(V,e_{n-1})=\mu(V,e_n)=0$. Hence, if $\mu_{23}^k=0$ for $k<n-1$, then $A\cong U\ltimes_{\phi}\kk^2$ for some $(n-2)$-dimensional vector space $U$ and some surjective skew-symmetric bilinear map $\phi:U\times U\rightarrow \kk^2$. Suppose now that $\mu_{23}^k\not=0$ for some $k<n-1$. Changing the basis, we may assume that $\mu(e_2,e_3)=e_{n-2}$. Since $(L_{e_2})^2=(L_{e_3})^2=0$, we have $\mu(e_2,e_{n-2})=\mu(e_3,e_{n-2})=0$.
Since $\dim \Im L_{e_2}\le 2$, $\mu(e_2,e_1)=-e_{n-1}$ and $\mu(e_2,e_3)=e_{n-2}$, we have $\Im L_{e_2}=\langle e_{n-2},e_{n-1}\rangle$. Analogously, $\Im L_{e_3}=\langle e_{n-2},e_{n}\rangle$.
Then we have $\mu(e_2,e_i)=\mu_{2,i}^{n-1}e_{n-1}$ and $\mu(e_3,e_i)=\mu_{3,i}^{n}e_{n}$ for $4\le i\le n-3$. Since $L_{e_2+e_3}(e_1)=e_{n-1}+e_n$, $L_{e_2+e_3}(e_3)=e_{n-2}$, $L_{e_2+e_3}(e_i)=\mu_{2,i}^{n-1}e_{n-1}+\mu_{3,i}^{n}e_n$ and $\dim \Im L_{e_2+e_3}\le 2$, one has $\mu_{2,i}^{n-1}=\mu_{3,i}^{n}$ for all $4\le i\le n-3$.

Replacing $e_i$ by $e_i+\mu_{2,i}^{n-1}e_1=e_i+\mu_{3,i}^{n}e_1$ for $4\le i\le n-3$, we may assume that $\mu(e_2,e_i)=\mu(e_3,e_i)=0$ for $4\le i\le n$.
Let us pick some $4\le i\le n-2$. Since $L_{e_2+e_i}(e_1)=-e_{n-1}$, $L_{e_2+e_i}(e_3)=e_{n-2}$ and $\dim \Im L_{e_2+e_i}\le 2$, one has $\mu(e_i,e_j)\subset \langle e_{n-1}\rangle$ for all $4\le j\le n-2$. Considering $L_{e_3+e_i}$, we get also $\mu(e_i,e_j)\subset \langle e_{n}\rangle$. Thus, $\mu(e_i,e_j)=0$ for $4\le i,j\le n-2$, and hence all nonzero products of basic elements are $\mu(e_1,e_2)=e_{n-1}$, $\mu(e_1,e_3)=e_{n}$ and $\mu(e_2,e_3)=e_{n-2}$, i.e. $\mu=T^{2,2}({\scriptstyle\epsilon_{23}^{n-2}})$.
\end{proof}

The classification of algebras of the form $U\ltimes_{\phi}\kk^2$ is strongly related to the classification of skew-symmetric matrix pairs considered, for example,  in  \cite{BGSS, Gant, S}. In fact, one has to factorize the classification obtained in these papers by an action of the group $GL(\kk^2)$. In terms of the algebra $U\ltimes_{\phi}\kk^2$, this action is defined by the equality $g*\left(U\ltimes_{\phi}\kk^2\right)=U\ltimes_{g\phi}\kk^2$ for $g\in GL(\kk^2)$. All the mentioned works consider the case $\charr\kk\not=2$ while the more complicated characteristic two case is considered in \cite{Wat} in a little more general settings than here. The deformation theory of skew-symmetric matrix pairs was considered in \cite{DK}. In our settings this problem differs a little but it still seems to be possible to give the general criteria of degenerations between algebras of the form $U\ltimes_{\phi}\kk^2$. In the current paper we will not solve this general problem and restrict us to the classifications of such algebras having level not greater than five. To do this we introduce the list of algebras below.

\begin{center}
Table 3. {\it Algebras of skew-symmetric matrix pairs.}\\
\begin{tabular}{|l|l|l|}
\hline
{\rm notation}&{\rm multiplication table}&{\rm dimension}\\
\hline
$\begin{array}{l}T^{2,2}({\scriptstyle\epsilon_{24}^n})\end{array}$&$\begin{array}{l}e_1e_2=e_{n-1},\,\,e_1e_3=e_2e_4=e_n\end{array}$&$\begin{array}{l}n\ge 6\end{array}$\\
\hline
$\begin{array}{l}T^{2,2}({\scriptstyle\epsilon_{34}^{n}})\end{array}$&$\begin{array}{l}e_1e_2=e_{n-1},\,\,e_1e_3=e_3e_4=e_n\end{array}$&$\begin{array}{l}n\ge 6\end{array}$\\
\hline
$\begin{array}{l}T^{2,2}({\scriptstyle\epsilon_{45}^n})\end{array}$&$\begin{array}{l}e_1e_2=e_{n-1},\,\,e_1e_3=e_4e_5=e_n\end{array}$&$\begin{array}{l}n\ge 7\end{array}$\\
\hline
\end{tabular}
\end{center}

Note that $T^{2,2}({\scriptstyle\epsilon_{34}^{n}})\cong {\bf n}_3\oplus {\bf n}_3$. To show this one has to simply replace $e_1$ by $e_1+e_4$. It was stated in \cite{gorb93} that this algebra has level three. We will show that in fact it has level four while this result is not new, see, for example, \cite{S90}.

\begin{lemma}\label{T22deg} One has
$
T^{2,2}({\scriptstyle\epsilon_{45}^n})\xrightarrow{\not\cong} T^{2,2}({\scriptstyle\epsilon_{34}^{n}})\xrightarrow{\not\cong} T^{2,2}({\scriptstyle\epsilon_{24}^n})\xrightarrow{\not\cong} T^{2,2}$. In particular,  $lev\big(T^{2,2}({\scriptstyle\epsilon_{45}^n})\big)\ge 5$.
\end{lemma}
\begin{proof} Since $IW_{e_1}\big(T^{2,2}({\scriptstyle\epsilon_{24}^n})\big)=T^{2,2}$ and $T^{2,2}({\scriptstyle\epsilon_{24}^n})\not\cong T^{2,2}$, we have $T^{2,2}({\scriptstyle\epsilon_{24}^n})\xrightarrow{\not\cong} T^{2,2}$.

Let us now construct the remaining degenerations. One has
$$
 T^{2,2}({\scriptstyle\epsilon_{45}^n})\xrightarrow{e_1,e_2,e_3+e_4,e_5,te_4,e_6,\dots,e_n} T^{2,2}({\scriptstyle\epsilon_{34}^{n}})
\xrightarrow{e_1,e_2+e_3,te_3,te_4,e_5\dots,e_{n-1}+e_n,te_n} T^{2,2}({\scriptstyle\epsilon_{24}^n}).
$$

Since
$
T^{2,2}({\scriptstyle\epsilon_{24}^n})\not\to_{(1,3,n),(3,3,n+1)}T^{2,2}({\scriptstyle\epsilon_{34}^{n}}) \not\to_{(1,5,n+1)} T^{2,2}({\scriptstyle\epsilon_{45}^n}),
$
the constructed degenerations are non-trivial.
\end{proof}

\begin{lemma}\label{T22rest} Suppose that $A\cong U\ltimes_{\phi}\kk^2$ for some $(n-2)$-dimensional vector space $U$ and some surjective skew-symmetric bilinear map $\phi:U\times U\rightarrow \kk^2$. If $A$ cannot be represented by $T^{2,2}$, $T^{2,2}({\scriptstyle\epsilon_{24}^n})$, $T^{2,2}({\scriptstyle\epsilon_{34}^{n}})$ or $T^{2,2}({\scriptstyle\epsilon_{45}^n})$, then $A\xrightarrow{\not\cong} T^{2,2}({\scriptstyle\epsilon_{45}^n})$ and, in particular, $lev(A)\ge 6$.
\end{lemma}
\begin{proof} We may assume that $A$ is represented by a structure $\mu$ such that $\mu(e_1,e_2)=e_{n-1}$, $\mu(e_1,e_3)=e_n$ and $\mu(e_1,e_i)=0$ for all $i\ge 3$. Replacing $e_2$ by $e_2-\mu_{2,3}^ne_1$ and $e_3$ by $e_3+\mu_{2,3}^{n-1}$, we may assume also that $\mu(e_2,e_3)=0$. If there exist $4\le i,j\le n-2$ such that $\mu(e_i,e_j)\not=0$, then we set
$\kappa_s(t):=\begin{cases}
1,&\mbox{ if $s=1$},\\
t,&\mbox{ if $s\in\{i,j\}$},\\
t^2,&\mbox{ otherwise.}
\end{cases}$

We have the degeneration
$
\mu\xrightarrow{\kappa_1(t)e_1,\dots,\kappa_n(t)e_n}T^{2,2}({\scriptstyle\mu_{i,j}^{n-1}\epsilon_{i,j}^{n-1}+\mu_{i,j}^{n}\epsilon_{i,j}^{n}})
$, where $T^{2,2}({\scriptstyle\mu_{i,j}^{n-1}\epsilon_{i,j}^{n-1}+\mu_{i,j}^{n}\epsilon_{i,j}^{n}})$ is the algebra with the multiplication table $e_1e_2=e_{n-1}$, $e_1e_3=e_n$, $e_ie_j=\mu_{i,j}^{n-1}e_{n-1}+\mu_{i,j}^{n}e_{n}$. It is clear that $T^{2,2}({\scriptstyle\mu_{i,j}^{n-1}\epsilon_{i,j}^{n-1}+\mu_{i,j}^{n}\epsilon_{i,j}^{n}})\cong T^{2,2}({\scriptstyle\epsilon_{45}^{n}})$.

If $\mu(e_i,e_j)=0$ for all $4\le i,j\le n-2$, then $\mu$ is determined by the matrices $M_i=\begin{pmatrix}\mu_{2,i}^{n-1}&\mu_{3,i}^{n-1}\\\mu_{2,i}^{n}&\mu_{3,i}^{n}\end{pmatrix}$ ($4\le i\le n-2$).
Note that $T^{2,2}({\scriptstyle\epsilon_{45}^{n}})$ is isomorphic to the algebra determined by matrices $M_4=\begin{pmatrix}0&0\\1&0\end{pmatrix}$, $M_5=\begin{pmatrix}0&0\\0&1\end{pmatrix}$ and $M_i=0$ for $6\le i\le n-2$. To see this, it is enough to calculate the structure constants of $T^{2,2}({\scriptstyle\epsilon_{45}^{n}})$ in the basis $-e_2-e_5,e_1,e_4,e_3,e_5,\dots,e_n$. Let $T^{2,2}({\scriptstyle\epsilon_{24}^{n}+\epsilon_{35}^{n}})\cong T^{2,2}({\scriptstyle\epsilon_{45}^{n}})$ denote the structure corresponding to the collection of the matrices defined above.

Replacing $e_i$ by $e_i+\alpha_i e_1$, we can replace our collection of matrices by the collection $M_i-\alpha_iE$ for any $\alpha_i\in\kk$ ($4\le i\le n-2$), where $E$ denotes the matrix of the identity map.
We also can apply any linear transformation to the elements $e_4,\dots,e_{n-2}$ that will induce the corresponding linear transformation of our collection of matrices.
Then we may assume that, for some $3\le r\le n-2$, $M_i=0$ for $i>r$ and the matrices $E,M_4,\dots,M_r$ are linearly independent, in particular, $r\le 6$, where the case $r=3$ corresponds to the structure $T^{2,2}$.
Replacing $e_2$ by $\alpha_{2,2}e_2+\alpha_{2,3}e_3$, $e_3$ by $\alpha_{3,2}e_2+\alpha_{3,3}e_3$, $e_{n-1}$ by $\alpha_{2,2}e_{n-1}+\alpha_{2,3}e_n$ and $e_n$ by $\alpha_{3,2}e_{n-1}+\alpha_{3,3}e_n$, where $S=\begin{pmatrix}\alpha_{2,2}&\alpha_{3,2}\\\alpha_{2,3}&\alpha_{3,3}\end{pmatrix}$ is a nonsingular matrix,  we can conjugate all matrices $M_i$ simultaneously by $S$.

If the number $r$ above equals to $6$, then the matrices $M_4$, $M_5$ and $M_6$ can be turned to any triple of matrices such that $E$, $M_4$, $M_5$ and $M_6$ are linearly independent. In particular we may assume that  $M_4=\begin{pmatrix}0&0\\1&0\end{pmatrix}$, $M_5=\begin{pmatrix}0&0\\0&1\end{pmatrix}$ and get the degeneration $\mu\xrightarrow{e_1,\dots,e_5,te_6,e_7,\dots,e_n}T^{2,2}({\scriptstyle\epsilon_{24}^{n}+\epsilon_{35}^{n}})$.

If the number $r$ above equals to $4$, then we choose some eigenvalue $\gamma$ of $M_4$ and replace $M_4$ by $M_4-\gamma E$. If after this $M_4$ has some nonzero eigenvalue, we rescale it to turn this value to $1$. Finally, we transform $M_4$ to its Jordan normal form and get either $M_4=\begin{pmatrix}0&0\\1&0\end{pmatrix}$ or $M_4=\begin{pmatrix}0&0\\0&1\end{pmatrix}$, i.e. $A$ can be represented by $T^{2,2}({\scriptstyle\epsilon_{24}^n})$ or $T^{2,2}({\scriptstyle\epsilon_{34}^{n}})$.

It remains to consider the case $r=5$. Using the transformations described above, we can turn $M_4$ either to the matrix $\begin{pmatrix}0&0\\1&0\end{pmatrix}$ or to the matrix $\begin{pmatrix}0&0\\0&1\end{pmatrix}$. Let us consider these two possibilities separately
\begin{enumerate}
\item Suppose that $M_4=\begin{pmatrix}0&0\\0&1\end{pmatrix}$.  We may assume that $M_5=\begin{pmatrix}0&\alpha\\\beta&0\end{pmatrix}$ for some $\alpha,\beta\in\kk$ not both zero. If $\beta=0$, then, replacing $M_4$ and $M_5$ by $\frac{1}{\alpha}SM_5S$ and $S(E-M_4)S$, where $S=\begin{pmatrix}0&1\\1&0\end{pmatrix}$, we again get the structure $T^{2,2}({\scriptstyle\epsilon_{24}^{n}+\epsilon_{35}^{n}})$. If $\beta\not=0$, then
$\mu\xrightarrow{e_1,e_2,te_3,\frac{t}{\beta}e_5,e_4,e_6,\dots,e_{n-1},te_n}T^{2,2}({\scriptstyle\epsilon_{24}^{n}+\epsilon_{35}^{n}})$.

\item Suppose that $M_4=\begin{pmatrix}0&0\\1&0\end{pmatrix}$. We may assume that $M_5=\begin{pmatrix}0&\alpha\\0&\beta\end{pmatrix}$ for some $\alpha,\beta\in\kk$ not both zero. If $\beta\not=0$, then we can turn $M_5$ to the form $M_5=\begin{pmatrix}0&0\\0&1\end{pmatrix}$, interchange $M_4$ and $M_5$, and return to the previous case. If $\beta=0$, then we may assume that $M_5=\begin{pmatrix}0&1\\0&0\end{pmatrix}$. Then
$\mu\xrightarrow{e_3,e_5,-e_1,te_2,\frac{1}{t}e_4,e_6,\dots,e_n}T^{2,2}({\scriptstyle\epsilon_{45}^{n}})$.
\end{enumerate}
\end{proof}

\begin{lemma}\label{T22lev} Suppose that $IW_1^{max}(A)\cong T^{2,2}$. Then
\begin{itemize}
\item $lev(A)=3$ if and only if $A$ can be represented by $T^{2,2}({\scriptstyle\epsilon_{23}^{n-2}})$ or $T^{2,2}({\scriptstyle\epsilon_{24}^n})$;
\item $lev(A)=4$ if and only if $A$ can be represented by $T^{2,2}({\scriptstyle\epsilon_{34}^{n}})$;
\item $lev(A)=5$ if and only if $A$ can be represented by $T^{2,2}({\scriptstyle\epsilon_{45}^{n}})$.
\end{itemize}
\end{lemma}
\begin{proof} By Corollary \ref{T22class} either $A$ can be represented by $T^{2,2}({\scriptstyle\epsilon_{23}^{n-2}})$ or $A\cong U\ltimes_{\phi}\kk^2$ for some $(n-2)$-dimensional vector space $U$ and some surjective skew-symmetric bilinear map $\phi:U\times U\rightarrow \kk^2$. In the last mentioned case if $A$ cannot be represented by $T^{2,2}$, $T^{2,2}({\scriptstyle\epsilon_{24}^n})$, $T^{2,2}({\scriptstyle\epsilon_{34}^{n}})$ or $T^{2,2}({\scriptstyle\epsilon_{45}^n})$, then $lev(A)\ge 6$ by Corollary \ref{T22rest}. If $T^{2,2}$ represents $A$, then $lev(A)=2$.

Suppose that $A$ can be represented by $T^{2,2}({\scriptstyle\epsilon_{23}^{n-2}})$ or $T^{2,2}({\scriptstyle\epsilon_{24}^n})$.
Since $T^{2,2}({\scriptstyle\epsilon_{23}^{n-2}})$ and $T^{2,2}({\scriptstyle\epsilon_{24}^n})$ degenerate to $T^{2,2}$, we have $lev(A)\ge 3$. Note that $T^{2,2}({\scriptstyle\epsilon_{23}^{n-2}})\not\to_{(1,4,n+1)} T^{2,2}({\scriptstyle\epsilon_{24}^n})$ and $T^{2,2}({\scriptstyle\epsilon_{24}^n})\not\to_{(1,1,n-1)} T^{2,2}({\scriptstyle\epsilon_{23}^{n-2}})$. If $lev(A)>3$, then $A$ degenerates to some algebra $B$ of level three. Since $IW_1^{max}(A)\to IW_1^{max}(B)$, we have either $IW_1^{max}(B)\cong T^{2,2}$ or $IW_1^{max}(B)\cong {\bf n}_3$. In the first case $B\cong A$, because all algebras with maximal one-dimensional IW contraction $T^{2,2}$ except $T^{2,2}$, $T^{2,2}({\scriptstyle\epsilon_{23}^{n-2}})$ and $T^{2,2}({\scriptstyle\epsilon_{24}^n})$ have level not less than four by Lemmas \ref{T22deg} and \ref{T22rest}. If $IW_1^{max}(B)\cong {\bf n}_3$, then $B$ can be represented by $\eta_3$ that contradicts the assertions $T^{2,2}({\scriptstyle\epsilon_{23}^{n-2}})\not\to_{(1,6,n+1)} \eta_3$ and $T^{2,2}({\scriptstyle\epsilon_{24}^n})\not\to_{(1,6,n+1)}\eta_3$. Hence,  $T^{2,2}({\scriptstyle\epsilon_{23}^{n-2}})$ and $T^{2,2}({\scriptstyle\epsilon_{24}^n})$ have level three.

The same argument shows that $lev\big(T^{2,2}({\scriptstyle\epsilon_{34}^{n}})\big)=4$ and $lev\big(T^{2,2}({\scriptstyle\epsilon_{45}^{n}})\big)=5$.
\end{proof}

\section{Algebras with maximal IW contraction of the form $T^{2,\dots,2}$}

This section is devoted to the classification of algebras $A$ such that either $IW_1^{max}(A)\cong T^{2,2,2}$ or $IW_1^{max}(A)\cong T^{2,2,2,2}$. We start with some general observations about degenerations of algebras of the form $T^{\overbrace{\scriptstyle 2,\dots,2}^{m}}$ with $m\ge 3$. Analogously to the case $m=2,3,4$, this algebra has dimension $n\ge 2m+1$ and the multiplication table defined by the equalities $e_1e_{i+1}=e_{i+n-m}$ ($1\le i\le m$).

\begin{lemma}\label{T2k2rest1} Let $A$ be an $n$-dimensional algebra with $IW_1^{max}(A)\cong T^{\overbrace{\scriptstyle 2,\dots,2}^{m}}$ for some $m\ge 3$. If $dim\,A^2> m$, then $A$ degenerates to one of the algebras
\begin{center}
\begin{tabular}{|l|l|}
\hline
$\begin{array}{l}T^{\overbrace{\scriptstyle 2,\dots,2}^{m}}({\scriptstyle\epsilon_{23}^{n-m}})\end{array}$&$\begin{array}{l}e_1e_{i+1}=e_{i+n-m},\,\,1\le i\le m,\,\,e_2e_3=e_{n-m}\end{array}$\\
\hline
$\begin{array}{l}T^{\overbrace{\scriptstyle 2,\dots,2}^{m}}({\scriptstyle\epsilon_{23}^{m+1}-\epsilon_{2,2+n-m}^{n}+\epsilon_{3,1+n-m}^{n}})\end{array}$&$\begin{array}{l}e_1e_{i+1}=e_{i+n-m},\,\,1\le i\le m,\\
e_2e_3=e_{m+1},\,\,e_2e_{2+n-m}=-e_n,\,\,e_3e_{1+n-m}=e_n\end{array}$\\
\hline
\end{tabular}
\end{center}
Moreover, in the case $n>2m+1$, one also has $T^{\overbrace{\scriptstyle 2,\dots,2}^{m}}({\scriptstyle\epsilon_{23}^{m+1}-\epsilon_{2,2+n-m}^{n}+\epsilon_{3,1+n-m}^{n}})\to T^{\overbrace{\scriptstyle 2,\dots,2}^{m}}({\scriptstyle\epsilon_{23}^{m+2}})$.
\end{lemma}
\begin{proof} Due to Corollary \ref{2k2sf}, $A$ can be represented by a structure $\mu$ such that $\mu(e_1,e_{i+1})=e_{i+n-m}$ for $1\le i\le m$,  $\mu(e_1,e_i)=0$ for $m+2\le i\le n$ and $\mu_{i,j}^k=0$ if $k\le max(i,j)$.

If $dim\,A^2> m$, then $\mu_{i,j}^k\not=0$ for some $1\le i<j<k\le n-m$. Due to Lemma \ref{2k2}, we have $2\le i<j\le m+1$. Let us consider two cases.
\begin{enumerate}
\item If $k> m+1$, then we set $\kappa_s(t):=\begin{cases}
1,&\mbox{ if $s=1$},\\
t^2,&\mbox{ if $s\in\{i,j,i+n-m-1,j+n-m-1\}$},\\
\mu_{i,j}^kt^4,&\mbox{ if $s=k$},\\
t^3,&\mbox{ otherwise.}
\end{cases}$

Due to Lemma \ref{2k2}, we have the degeneration
$
\mu\xrightarrow{\kappa_1(t)e_1,\dots,\kappa_n(t)e_n}T^{\overbrace{\scriptstyle 2,\dots,2}^{m}}({\scriptstyle\epsilon_{i,j}^{k}})
$, where $T^{\overbrace{\scriptstyle 2,\dots,2}^{m}}({\scriptstyle\epsilon_{i,j}^{k}})$ is the algebra with the multiplication table $e_1e_{s+1}=e_{s+n-m}$ ($1\le s\le m$), $e_ie_j=e_k$. It is clear that $T^{\overbrace{\scriptstyle 2,\dots,2}^{m}}({\scriptstyle\epsilon_{i,j}^{k}})\cong T^{\overbrace{\scriptstyle 2,\dots,2}^{m}}({\scriptstyle\epsilon_{23}^{n-m}})$.

\item If $k\le m+1$, then we set $\kappa_s(t):=\begin{cases}
1,&\mbox{ if $s=1$},\\
t^2,&\mbox{ if $s\in\{i,j,i+n-m-1,j+n-m-1\}$},\\
\mu_{i,j}^kt^4,&\mbox{ if $s\in\{k,k+n-m-1\}$},\\
t^3,&\mbox{ otherwise.}
\end{cases}$

Due to Lemma \ref{2k2}, we have the degeneration
$$
\mu\xrightarrow{\kappa_1(t)e_1,\dots,\kappa_n(t)e_n}T^{\overbrace{\scriptstyle 2,\dots,2}^{m}}({\scriptstyle\epsilon_{i,j}^{k}+\mu_{i,j}^{k+n-m-1}\epsilon_{i,j}^{k+n-m-1}-\epsilon_{i,j+n-m-1}^{k+n-m-1}+\epsilon_{j,i+n-m-1}^{k+n-m-1}}),
$$ where $T^{\overbrace{\scriptstyle 2,\dots,2}^{m}}({\scriptstyle\epsilon_{i,j}^{k}+\mu_{i,j}^{k+n-m-1}\epsilon_{i,j}^{k+n-m-1}-\epsilon_{i,j+n-m-1}^{k+n-m-1}+\epsilon_{j,i+n-m-1}^{k+n-m-1}})$ is the algebra with the multiplication table 
\begin{multline*}
e_1e_{s+1}=e_{s+n-m}\,\,(1\le s\le m),\,\,e_ie_j=e_k+\mu_{i,j}^{k+n-m-1}e_{k+n-m-1},\\e_ie_{j+n-m-1}=-e_{k+n-m-1},\,\,e_je_{i+n-m-1}=e_{k+n-m-1}.
\end{multline*} It is clear that $$T^{\overbrace{\scriptstyle 2,\dots,2}^{m}}({\scriptstyle\epsilon_{i,j}^{k}+\mu_{i,j}^{k+n-m-1}\epsilon_{i,j}^{k+n-m-1}-\epsilon_{i,j+n-m-1}^{k+n-m-1}+\epsilon_{j,i+n-m-1}^{k+n-m-1}})\cong T^{\overbrace{\scriptstyle 2,\dots,2}^{m}}({\scriptstyle\epsilon_{23}^{m+1}-\epsilon_{2,2+n-m}^{n}+\epsilon_{3,1+n-m}^{n}}).$$
\end{enumerate}

To finish the proof, it remains to note that if $n>2m+1$, then
$$
T^{\overbrace{\scriptstyle 2,\dots,2}^{m}}({\scriptstyle\epsilon_{23}^{m+1}-\epsilon_{2,2+n-m}^{n}+\epsilon_{3,1+n-m}^{n}})
\xrightarrow{ e_1,\dots,e_m,\frac{1}{t}e_{m+1}-\frac{1}{t}e_{n-m},e_{m+2},\dots,e_{n-1},\frac{1}{t}e_n}
T^{\overbrace{\scriptstyle 2,\dots,2}^{m}}({\scriptstyle\epsilon_{23}^{n-m}}). 
$$
\end{proof}

\begin{lemma}\label{T2k2rest2} Let $A$ be an $n$-dimensional algebra with $IW_1^{max}(A)\cong T^{\overbrace{\scriptstyle 2,\dots,2}^{m}}$ for some $m\ge 3$ and $dim\,A^2=m$.
If $dim\,\Ann(A)<n-m-1$, then $A$ degenerates to the algebra
\begin{center}
\begin{tabular}{|l|l|}
\hline
$\begin{array}{l}T^{\overbrace{\scriptstyle 2,\dots,2}^{m}}({\scriptstyle\epsilon_{2,m+2}^n})\end{array}$&$\begin{array}{l}e_1e_{i+1}=e_{i+n-m},\,\,1\le i\le m,\,\,e_2e_{m+2}=e_n\end{array}$\\
\hline
\end{tabular}
\end{center}
\end{lemma}
\begin{proof} We again represent $A$ by a structure $\mu$ such that $\mu(e_1,e_{s+1})=e_{s+n-m}$ for $1\le s\le m$,  $\mu(e_1,e_s)=0$ for $m+2\le s\le n$ and $\mu_{i,j}^k=0$ if $k\le max(i,j)$. Since $dim\,A^2=m$, we also have $\mu_{i,j}^k=0$ if either $k\le n-m$ or $max(i,j)>n-m$ by Lemma \ref{2k2}.
Replacing $e_s$ by $e_s+\mu_{2,s}^{n-m+1}e_1$, we may assume that $\mu_{2,s}^{n-m+1}=0$ for all $m+2\le s\le n-m$.
Since $dim\,\Ann(A)<n-m-1$, we have $\mu_{i,j}^k\not=0$ for some $2\le i,j,k\le n$ with $j\ge m+2$. 
By our assumptions, we have $n-m+1\le k\le n$. We also may assume that $2\le i\le m+1$. Indeed, if $i\ge m+2$, then either $\mu_{3,j}^k\not=0$ and we can take replace $i$ by $3$ or $\mu_{3,j}^k=0$ and then we can first replace $e_3$ by $e_3+e_i$ and after that $i$ by $3$.

If $k\not=i+n-m-1$, then we set $\kappa_s(t):=\begin{cases}
1,&\mbox{ if $s=1$},\\
t^2,&\mbox{ if $s\in\{i,j,i+n-m-1\}$},\\
\mu_{i,j}^kt^4,&\mbox{ if $s\in\{k,k-n+m+1\}$},\\
t^3,&\mbox{ otherwise.}
\end{cases}$

We have the degeneration
$
\mu\xrightarrow{\kappa_1(t)e_1,\dots,\kappa_n(t)e_n}T^{\overbrace{\scriptstyle 2,\dots,2}^{m}}({\scriptstyle\epsilon_{i,j}^{k}}),
$ where $T^{\overbrace{\scriptstyle 2,\dots,2}^{m}}({\scriptstyle\epsilon_{i,j}^{k}})$ is the algebra with the multiplication table 
$e_1e_{s+1}=e_{s+n-m}$ ($1\le s\le m$), $e_ie_j=e_k$. It is clear that
$T^{\overbrace{\scriptstyle 2,\dots,2}^{m}}({\scriptstyle\epsilon_{i,j}^{k}})\cong T^{\overbrace{\scriptstyle 2,\dots,2}^{m}}({\scriptstyle\epsilon_{2,m+2}^{n}}).$

If $k=i+n-m-1$, then we have $i\not=2$ and there is $\alpha\in\kk$ such that $\mu_{i,j}^{i+n-m-1}\alpha^2\not=\mu_{i,j}^k\alpha+\mu_{2,j}^{i+n-m-1}$. Then replacing $e_2$ and $e_{n-m+1}$ by $e_2+\alpha e_i$ and $e_{n-m+1}+\alpha e_{k}$, we may assume that $\mu_{2,j}^{k}\not=0$ and return to the case that we have already considered.
\end{proof}

\begin{lemma}\label{T2k2rest3} Let $A$ be an $n$-dimensional algebra with $IW_1^{max}(A)\cong T^{\overbrace{\scriptstyle 2,\dots,2}^{m}}$ for some $m\ge 3$. If $A\not\cong T^{\overbrace{\scriptstyle 2,\dots,2}^{m}}$, then $A$ degenerates to the algebra
\begin{center}
\begin{tabular}{|l|l|}
\hline
$\begin{array}{l}T^{\overbrace{\scriptstyle 2,\dots,2}^{m}}({\scriptstyle\epsilon_{23}^n})\end{array}$&$\begin{array}{l}e_1e_{i+1}=e_{i+n-m},\,\,1\le i\le m,\,\,e_2e_3=e_n\end{array}$\\
\hline
\end{tabular}
\end{center}
\end{lemma}
\begin{proof} If $dim\,A^2>m$, then $A\to T^{\overbrace{\scriptstyle 2,\dots,2}^{m}}({\scriptstyle\epsilon_{23}^{m+1}-\epsilon_{2,2+n-m}^{n}+\epsilon_{3,1+n-m}^{n}})$ or $A\to T^{2,2,2}({\scriptstyle\epsilon_{23}^{n-m}})$ by Lemma \ref{T2k2rest1}.
Since $T^{\overbrace{\scriptstyle 2,\dots,2}^{m}}({\scriptstyle\epsilon_{23}^{m+1}-\epsilon_{2,2+n-m}^{n}+\epsilon_{3,1+n-m}^{n}})\xrightarrow{e_1,\dots,e_m,\frac{1}{t}e_{m+1}-\frac{1}{t^2}e_{n},e_{m+2},\dots,e_{n-1},\frac{1}{t}e_n}T^{\overbrace{\scriptstyle 2,\dots,2}^{m}}({\scriptstyle\epsilon_{23}^n})$ and
$T^{\overbrace{\scriptstyle 2,\dots,2}^{m}}({\scriptstyle\epsilon_{23}^{n-m}})\xrightarrow{e_1,\dots,e_{n-n-1},\frac{1}{t}e_{n-m}-\frac{1}{t}e_n,e_{n-m+1},\dots,e_n}T^{2,2,2}({\scriptstyle\epsilon_{23}^n})$, we may assume that $dim\,A^2=m$.

If $dim\,\Ann(A)<n-m-1$, then $A\to T^{\overbrace{\scriptstyle 2,\dots,2}^{m}}({\scriptstyle\epsilon_{2,m+2}^n})$ by Lemma \ref{T2k2rest2}. Since 
$$T^{\overbrace{\scriptstyle 2,\dots,2}^{m}}({\scriptstyle\epsilon_{2,m+2}^n})\xrightarrow{e_1,e_2,e_3+e_{m+2},e_4,\dots,e_{m+1},te_{m+2},e_{m+3},\dots,e_n} T^{\overbrace{\scriptstyle 2,\dots,2}^{m}}({\scriptstyle\epsilon_{23}^n}),$$
 it remains to consider the case where $A$ is represented by a structure $\mu$ such that $\mu(e_1,e_{s+1})=e_{s+n-3}$ for $1\le s\le m$ and $\mu_{i,j}^k=0$ if either $max(i,j)>m+1$ or $k<n-m+1$. We will consider three cases.
\begin{enumerate}
\item There are some pairwise different $2\le i,j,k\le m+1$ such that $\mu_{i,j}^{k+n-m-1}\not=0$. Then we set $\kappa_s(t):=\begin{cases}
1,&\mbox{ if $s=1$},\\
t^2,&\mbox{ if $s\in\{i,j,i+n-m-1,j+n-m-1\}$},\\
\mu_{i,j}^{k+n-m-1}t^4,&\mbox{ if $s\in\{k,k-n+m+1\}$},\\
t^3,&\mbox{ otherwise.}
\end{cases}$

We have the degeneration
$
\mu\xrightarrow{\kappa_1(t)e_1,\dots,\kappa_n(t)e_n}T^{\overbrace{\scriptstyle 2,\dots,2}^{m}}({\scriptstyle\epsilon_{i,j}^{k+n-m-1}}),
$ where $T^{\overbrace{\scriptstyle 2,\dots,2}^{m}}({\scriptstyle\epsilon_{i,j}^{k+n-m-1}})$ is the algebra with the multiplication table 
$e_1e_{s+1}=e_{s+n-m}$ ($1\le s\le m$), $e_ie_j=e_{k+n-m-1}$. It is clear that
$T^{\overbrace{\scriptstyle 2,\dots,2}^{m}}({\scriptstyle\epsilon_{i,j}^{k+n-m-1}})\cong T^{\overbrace{\scriptstyle 2,\dots,2}^{m}}({\scriptstyle\epsilon_{23}^{n}}).$

\item For any pairwise different $2\le i,j,k\le m+1$ one has $\mu_{i,j}^{k+n-m-1}=0$, but there are pairwise different $2\le i,j,k\le m+1$ such that $\mu_{i,j}^{j+n-m-1}\not=\mu_{i,k}^{k+n-m-1}$. Then replacing $e_j$ and $e_{j+n-m+1}$ by $e_j+ e_k$ and $e_{j+n-m-1}+e_{k+n-m-1}$, we may assume that $\mu_{i,j}^{k+n-m-1}\not=0$ and return to the previous case.

\item There are $\alpha_i\in\kk$ ($2\le i\le m+1$) such that, for any $2\le i,j\le m+1$, one has $\mu(e_i,e_j)=\alpha_i e_{j+n-m-1}-\alpha_j e_{i+n-m-1}$. Replacing $e_i$ by $e_i-\alpha_i e_1$ for $2\le i\le m+1$, one sees that $\mu\cong T^{\overbrace{\scriptstyle 2,\dots,2}^{m}}$ in this case that contradicts our assumptions.
\end{enumerate}
Since the considered cases cover all possibilities, we are done.
\end{proof}

To fulfill the part of our classification announced in the beginning of this section, we will need the algebra structures presented in the next table.
\begin{center}
Table 4. {\it Algebras with $IW_1^{max}(A)=T^{2,2,2}$.}\\
\begin{tabular}{|l|l|l|}
\hline
{\rm notation}&{\rm multiplication table}&{\rm dimension}\\
\hline
$\begin{array}{l}T^{2,2,2}({\scriptstyle\epsilon_{23}^n})\end{array}$&$\begin{array}{l}e_1e_{i+1}=e_{i+n-3},\,\,1\le i\le 3,\,\,e_2e_3=e_n\end{array}$&$\begin{array}{l}n\ge 7\end{array}$\\
\hline
$\begin{array}{l}T^{2,2,2}({\scriptstyle\epsilon_{24}^n})\end{array}$&$\begin{array}{l}e_1e_{i+1}=e_{i+n-3},\,\,1\le i\le 3,\,\,e_2e_4=e_n\end{array}$&$\begin{array}{l}n\ge 7\end{array}$\\
\hline
$\begin{array}{l}T^{2,2,2}({\scriptstyle\epsilon_{23}^4-\epsilon_{26}^7+\epsilon_{35}^7})\end{array}$&$\begin{array}{l}e_1e_{i+1}=e_{i+4},\,\,1\le i\le 3,\\e_2e_3=e_4,\,\,e_2e_6=-e_7,\,\,e_3e_5=e_7\end{array}$&$\begin{array}{l}n=7\end{array}$\\
\hline
\end{tabular}
\end{center}

It will follow from Lemma \ref{T2k2rest1} and what we will prove later that $lev_{\infty}\big(T^{2,2,2}({\scriptstyle\epsilon_{23}^4-\epsilon_{26}^7+\epsilon_{35}^7})\big)\ge 7$. In contrast to this, in dimension $7$ the algebra $T^{2,2,2}({\scriptstyle\epsilon_{23}^4-\epsilon_{26}^7+\epsilon_{35}^7})$ has level five. To prove this, we need to show that it does not degenerate to $T^{2,2}({\scriptstyle\epsilon_{45}^n})$ and $T^{2,2,2}({\scriptstyle\epsilon_{24}^n})$. Unfortunately we have not found some short prove of this fact, and so give a very tedious calculation proving it in the next lemma.

\begin{lemma}\label{ex222} In the variety $\mathcal{AC}_7$ one has $T^{2,2,2}({\scriptstyle\epsilon_{23}^4-\epsilon_{26}^7+\epsilon_{35}^7})\not\to T^{2,2}({\scriptstyle\epsilon_{45}^n}),T^{2,2,2}({\scriptstyle\epsilon_{24}^n})$.
\end{lemma}
\begin{proof} Let us consider the set
$$
\mathcal{R}=\left\{\lambda\in\mathcal{AC}_7\left| \begin{array}{c}
\lambda(V,V_7)+\lambda(V_2,V_6)+\lambda(V_3,V_5)=0,\\
\lambda(V,V_4)\subset V_7,\lambda(V_2,V_3)\subset V_6,\lambda(V,V_3)\subset V_5,\lambda(V,V)\subset V_4,\\
\lambda_{12}^4\lambda_{34}^7=\lambda_{23}^6\lambda_{16}^7,\lambda_{12}^4\lambda_{34}^7+\lambda_{13}^5\lambda_{25}^7=0,\lambda_{12}^5\lambda_{34}^7=\lambda_{13}^5\lambda_{24}^7,\\
\lambda_{12}^5\lambda_{25}^7+\lambda_{12}^4\lambda_{24}^7=0,\lambda_{23}^6\lambda_{15}^7=\lambda_{13}^6\lambda_{25}^7,\lambda_{13}^6\lambda_{16}^7+\lambda_{13}^5\lambda_{15}^7=0,\\
\lambda_{23}^6\lambda_{14}^7-\lambda_{13}^6\lambda_{24}^7+\lambda_{12}^6\lambda_{34}^7=0
\end{array}\right.\right\}.
$$
Direct verifications show that $\mathcal{R}$ is a closed subset of $\mathcal{A}_n$ invariant under lower triangular transformations. Considering the basis $e_1,e_2,e_3,e_5,e_6,e_4,e_7$, one sees that $\lambda\cap O\big(T^{2,2,2}({\scriptstyle\epsilon_{23}^4-\epsilon_{26}^7+\epsilon_{35}^7})\big)\not=\varnothing$. On the other hand, a direct calculation shows that
$$
\lambda\cap O\big(T^{2,2}({\scriptstyle\epsilon_{45}^n})\big)=\lambda\cap O\big(T^{2,2,2}({\scriptstyle\epsilon_{24}^n})\big)=\varnothing.
$$
We will fulfill this calculation in the more difficult case of the algebra $T^{2,2}({\scriptstyle\epsilon_{45}^n})$ and leave the second case to the reader.

Suppose that we have found some $\lambda\in\mathcal{R}$ and a basis $f_1,\dots,f_7$ of $V$ such that the structure constants of $\lambda$ in the basis $f_1,\dots,f_7$ are the same as the structure constants of $T^{2,2}({\scriptstyle\epsilon_{45}^n})$ in the basis $e_1,\dots,e_7$, i.e. $\lambda(f_1,f_2)=f_6$, $\lambda(f_1,f_3)=\lambda(f_4,f_5)=f_7$. Let us pick some $v=\sum\limits_{i=1}^7\alpha_if_i\in V_4$. Using the condition $\lambda(V,V_4)\subset V_7$, we see that $\alpha_2f_6+\alpha_3f_7$, $\alpha_1 f_6$, $\alpha_1f_7$, $\alpha_4 f_7$ and $\alpha_5 f_7$ belong to $\langle e_7\rangle$. If $\langle e_7\rangle\not= \langle f_7\rangle$, then we have $\alpha_1=\alpha_4=\alpha_5$ for any element of $V_4$, and hence the dimension argument implies $V_4=\langle f_2,f_3,f_6,f_7\rangle$. But in this case $\lambda(V,V_4)=\langle f_6,f_7\rangle$ and we get a construction. Thus, we may assume that $e_7=f_7$ and $\alpha_1=\alpha_2=0$ for any $v\in V_4$. Using the condition $V_4^2=0$, we see that $V_4=\langle f_3,\alpha f_4+\beta f_5, f_6,f_7\rangle$ for some $\alpha,\beta\in\kk$. Without loss of generality we may assume that $V_4=\langle f_3, f_5, f_6,f_7\rangle$ and $\langle e_1,e_2,e_3\rangle=\langle f_1,f_2,f_4\rangle$. Suppose first that $f_6\not\in V_5$. Since $\mathcal{R}$ is  invariant under lower triangular transformations, we may assume that $e_4=f_6$. Using the condition $\lambda(V,V_3)\subset V_5$, we get $e_3=f_4$. But in this case $\lambda(V_3,V_5)\not=0$ and we get a contradiction.

We get $f_6\not\in V_5$, and hence we may assume that $e_4\in \{f_3,f_5\}$. In particular, we have $\lambda_{12}^4=0$, and hence $\lambda_{12}^5\lambda_{25}^7=\lambda_{13}^5\lambda_{25}^7=0$. Now we have two cases.
\begin{enumerate}
\item $\lambda_{12}^5=\lambda_{13}^5=0$. With the conditions $\lambda(V_2,V_3)\subset V_6$ and $f_6\not\in V_5$ this implies that $\lambda(V,V)\subset V_6$, i.e. we may assume that $e_6=f_6$, $\langle e_4,e_5\rangle= \langle f_3,f_5\rangle$.
Suppose that $e_i=\alpha_{i,1}f_1+\alpha_{i,2}f_2+\alpha_{i,4}f_4$ for $1\le i\le 3$ and $e_i=\alpha_{i,3}f_3+\alpha_{i,5}f_5$ for $i=4,5$, where all $\alpha_{i,j}$ are from $\kk$. Rewriting the conditions $\lambda_{35}^7=0$, $\lambda_{23}^6\lambda_{15}^7=\lambda_{13}^6\lambda_{25}^7$ and $\lambda_{23}^6\lambda_{14}^7-\lambda_{13}^6\lambda_{24}^7+\lambda_{12}^6\lambda_{34}^7=0$ in terms of $\alpha_{i,j}$, we get
$$
\alpha_{31}\alpha_{53}+\alpha_{34}\alpha_{55}=0,\,\, (\alpha_{11}\alpha_{53}+\alpha_{14}\alpha_{55})\begin{vmatrix}\alpha_{21}&\alpha_{22}\\ \alpha_{31}&\alpha_{32}\end{vmatrix}=(\alpha_{21}\alpha_{53}+\alpha_{24}\alpha_{55})\begin{vmatrix}\alpha_{11}&\alpha_{12}\\ \alpha_{31}&\alpha_{32}\end{vmatrix}
$$
and
\begin{multline*}
(\alpha_{11}\alpha_{43}+\alpha_{14}\alpha_{45})\begin{vmatrix}\alpha_{21}&\alpha_{22}\\ \alpha_{31}&\alpha_{32}\end{vmatrix}-(\alpha_{21}\alpha_{43}+\alpha_{24}\alpha_{45})\begin{vmatrix}\alpha_{11}&\alpha_{12}\\ \alpha_{31}&\alpha_{32}\end{vmatrix}\\
+(\alpha_{31}\alpha_{43}+\alpha_{34}\alpha_{45})\begin{vmatrix}\alpha_{11}&\alpha_{12}\\ \alpha_{21}&\alpha_{22}\end{vmatrix}=0.
\end{multline*}
Noting that
$$
\alpha_{11}\begin{vmatrix}\alpha_{21}&\alpha_{22}\\ \alpha_{31}&\alpha_{32}\end{vmatrix}-\alpha_{21}\begin{vmatrix}\alpha_{11}&\alpha_{12}\\ \alpha_{31}&\alpha_{32}\end{vmatrix}
+\alpha_{31}\begin{vmatrix}\alpha_{11}&\alpha_{12}\\ \alpha_{21}&\alpha_{22}\end{vmatrix}=0
$$
and
$$
\alpha_{14}\begin{vmatrix}\alpha_{21}&\alpha_{22}\\ \alpha_{31}&\alpha_{32}\end{vmatrix}-\alpha_{24}\begin{vmatrix}\alpha_{11}&\alpha_{12}\\ \alpha_{31}&\alpha_{32}\end{vmatrix}
+\alpha_{34}\begin{vmatrix}\alpha_{11}&\alpha_{12}\\ \alpha_{21}&\alpha_{22}\end{vmatrix}=
\begin{vmatrix}\alpha_{11}&\alpha_{12}&\alpha_{14}\\ \alpha_{21}&\alpha_{22}&\alpha_{24}\\\alpha_{31}&\alpha_{32}&\alpha_{34}\end{vmatrix}\not=0,
$$
one sees that the equalities $\lambda_{35}^7=0$, $\lambda_{15}^7\lambda_{23}^6=\lambda_{13}^6\lambda_{25}^7$ imply $\alpha_{55}=0$ and the equality $\lambda_{14}^7\lambda_{23}^6-\lambda_{24}^7\lambda_{13}^6+\lambda_{34}^7\lambda_{12}^6=0$ implies $\alpha_{45}=0$ that contradicts the linear independence of $e_4$ and $e_5$.

\item $\lambda_{25}^7=0$. With the conditions $\lambda(V_2,V_6)+\lambda(V_3,V_5)=0$ and $\lambda(V,V_4)\subset V_7$ this implies that $\lambda(V_2,V_5)=0$. Note that $\lambda_{15}^7\lambda_{23}^6=\lambda_{16}^7\lambda_{23}^6=0$. If $\lambda_{15}^7=\lambda_{16}^7=0$, then, using other conditions assumed and obtained earlier, we get $\lambda(V,V_5)=0$ that is impossible. Hence, we have $\lambda_{23}^6=0$.  Since $\lambda(V_2,V_3)\subset V_6$ and $\lambda_{23}^7=0$, we have $\lambda(e_2,e_3)=0$.
Note that $f_2\in \langle e_2,e_3\rangle$ because in the opposite case we would have $\lambda_{15}^7=\lambda_{16}^7=0$ that have already been proved to be impossible.
Hence, $\langle e_2,e_3\rangle=\langle f_2,f_4\rangle$ and the condition $\lambda(V_2,V_5)=0$ gives also $\langle e_5,e_6\rangle=\langle f_3,f_6\rangle$. Then we also may assume that $e_1=f_1$ and $e_4=f_5$.
Let $\alpha,\beta\in\kk$ be such that $f_4=\alpha e_2+\beta e_3$. Then $\alpha\lambda_{12}^5+\beta\lambda_{13}^5=\alpha\lambda_{12}^6+\beta\lambda_{13}^6=0$ and $\alpha\lambda_{24}^7+\beta\lambda_{34}^7\not=0$.
Then the equalities $\lambda_{12}^5\lambda_{34}^7=\lambda_{13}^5\lambda_{24}^7$ and $\lambda_{23}^6\lambda_{14}^7-\lambda_{24}^7\lambda_{13}^6+\lambda_{12}^6\lambda_{34}^7=0$ imply $\lambda_{12}^5=\lambda_{13}^5=\lambda_{12}^6=\lambda_{13}^6=0$ that is impossible.
\end{enumerate}

The obtained contradiction shows that $\lambda\cap O\big(T^{2,2}({\scriptstyle\epsilon_{45}^n})\big)=\varnothing$.
\end{proof}

\begin{lemma}\label{T222lev} Suppose that $IW_1^{max}(A)\cong T^{2,2,2}$. Then
\begin{itemize}
\item  $lev(A)=4$ if and only if $A$ can be represented by $T^{2,2,2}({\scriptstyle\epsilon_{23}^n})$;
\item $lev(A)=5$ if and only if either $A$ can be represented by $T^{2,2,2}({\scriptstyle\epsilon_{24}^n})$ or $dim\,A=7$ and $A$  can be represented by $T^{2,2,2}({\scriptstyle\epsilon_{23}^4-\epsilon_{26}^7+\epsilon_{35}^7})$.
\end{itemize}
\end{lemma}
\begin{proof} We have $A\to T^{2,2,2}({\scriptstyle\epsilon_{23}^n})$ by Lemma \ref{T2k2rest3}. Since $T^{2,2,2}({\scriptstyle\epsilon_{23}^n})\xrightarrow{\not\cong}T^{2,2,2}$, we have $lev\big(T^{2,2,2}({\scriptstyle\epsilon_{23}^n})\big)\ge 4$. Hence, if  $lev(A)=4$, then $A$ can be represented by $T^{2,2,2}({\scriptstyle\epsilon_{23}^n})$. If $lev\big(T^{2,2,2}({\scriptstyle\epsilon_{23}^n})\big)>4$, then $T^{2,2,2}({\scriptstyle\epsilon_{23}^n})$ has to degenerate to some algebra $B$ with $lev(B)=4$. Since in this case $IW_1^{max}\big(T^{2,2,2}({\scriptstyle\epsilon_{23}^n})\big)\to IW_1^{max}(B)$, we have $IW_1^{max}(B)\cong {\bf n}_3$, $IW_1^{max}(B)\cong T^{2,2}$ or $IW_1^{max}(B)\cong T^{2,2,2}$. In the last case $B\cong T^{2,2,2}({\scriptstyle\epsilon_{23}^n})$ by Lemma \ref{T2k2rest3}. In the remaining cases $B$ is isomorphic to $\eta_4$ or $T^{2,2}({\scriptstyle\epsilon_{34}^{n}})$ (see Lemma \ref{T22lev}). Since $T^{2,2,2}({\scriptstyle\epsilon_{23}^n})\not\to_{(1,5,n+1)}\eta_4$ and $T^{2,2,2}({\scriptstyle\epsilon_{23}^n})\not\to_{(1,4,n),(2,3,n),(2,4,n+1)}T^{2,2}({\scriptstyle\epsilon_{34}^{n}})$, we have $lev\big(T^{2,2,2}({\scriptstyle\epsilon_{23}^n})\big)=4$.

Since both $T^{2,2,2}({\scriptstyle\epsilon_{24}^n})$ and $T^{2,2,2}({\scriptstyle\epsilon_{23}^4-\epsilon_{26}^7+\epsilon_{35}^7})$ degenerate non-trivially to $T^{2,2,2}({\scriptstyle\epsilon_{23}^n})$, both of these algebras have levels not less than five.

Let us show now that if $IW_1^{max}(A)\cong T^{2,2,2}$ and $A\not\cong T^{2,2,2},T^{2,2,2}({\scriptstyle\epsilon_{23}^n})$, then either $A\to T^{2,2,2}({\scriptstyle\epsilon_{24}^n})$ or $dim\,A=7$ and $A\to T^{2,2,2}({\scriptstyle\epsilon_{23}^4-\epsilon_{26}^7+\epsilon_{35}^7})$. If $dim\,A^2>3$, then either $dim\,A=7$ and $A\to T^{2,2,2}({\scriptstyle\epsilon_{23}^4-\epsilon_{26}^7+\epsilon_{35}^7})$ or $dim\,A>7$ and $A\to T^{2,2,2}({\scriptstyle\epsilon_{23}^{n-3}})$ by Lemma \ref{T2k2rest1}. Since $T^{2,2,2}({\scriptstyle\epsilon_{23}^{n-3}})\xrightarrow{e_1,e_2,e_4,te_3,e_5\dots,e_{n-4},e_{n-3}-e_{n-1},e_{n-2},e_n,te_{n-1}}T^{2,2,2}({\scriptstyle\epsilon_{24}^n})$, we may assume that $dim\,A^2=3$.

If $dim\,\Ann(A)<n-4$, then $A\to T^{ 2,2,2}({\scriptstyle\epsilon_{25}^n})$ by Lemma \ref{T2k2rest2}. Since 
$$T^{ 2,2,2}({\scriptstyle\epsilon_{25}^n})\xrightarrow{e_1,e_2,e_3,e_4+e_5,te_5,e_6,\dots,e_n} T^{2,2,2}({\scriptstyle\epsilon_{24}^n}),$$
it remains to consider the case where $A$ is represented by a structure $\mu$ such that $\mu(e_1,e_{i+1})=e_{i+n-3}$ for $1\le i\le 3$ and $\mu_{i,j}^k=0$ if either $max(i,j)>4$ or $k<n-2$.

We also may assume that $\mu_{23}^n\not=0$ (see the proof of Lemma \ref{T2k2rest3}). 
Replacing $e_4$ by $\mu_{23}^2e_2+\mu_{23}^3e_3+\mu_{23}^4e_4$ and $e_n$ by $\mu(e_2,e_3)$, we may assume that $\mu(e_2,e_3)=e_n$.
Let $\gamma$ be an eigenvalue of the matrix $\begin{pmatrix}\mu_{24}^{n-2}&\mu_{34}^{n-2}\\\mu_{24}^{n-1}&\mu_{34}^{n-1}\end{pmatrix}$. Then there are some $\alpha,\beta\in\kk$ not all zero such that $(\mu_{24}^{n-2}-\gamma)e_2+\mu_{24}^{n-1}e_3,\mu_{34}^{n-2}e_2+(\mu_{34}^{n-1}-\gamma)e_3\in\langle \alpha e_2+\beta e_3\rangle$. We may assume that $\beta\not=0$. Then replacing $e_3$ by $\alpha e_2+\beta e_3$, $e_{n-1}$ by $\alpha e_{n-2}+\beta e_{n-1}$ and $e_4$ by $e_4+\gamma e_1$, we may assume that $\mu_{24}^2=\mu_{34}^2=0$.

It remains to consider two cases.
\begin{enumerate}
\item If $\mu(e_3,e_4)\not=0$, then $\mu\xrightarrow{e_1,te_2,e_3,\dots,e_{n-3},te_{n-2},e_{n-1},e_n}T^{2,2,2}({\scriptstyle \mu_{34}^{n-1}\epsilon_{34}^{n-1}+\mu_{34}^n\epsilon_{34}^n})$, where  $T^{2,2,2}({\scriptstyle \mu_{34}^{n-1}\epsilon_{34}^{n-1}+\mu_{34}^n\epsilon_{34}^n})$ is the algebra with the multiplication table 
$e_1e_{s+1}=e_{s+n-3}$ ($1\le s\le 3$), $e_3e_4=\mu_{34}^{n-1}e_{n-1}+\mu_{34}^{n}e_n$. It is clear that
$T^{2,2,2}({\scriptstyle \mu_{34}^{n-1}\epsilon_{34}^{n-1}+\mu_{34}^n\epsilon_{34}^n})\cong T^{2,2,2}({\scriptstyle \epsilon_{24}^n}).$
\item If $\mu(e_3,e_4)=0$, then we consider the operator $L_{e_2}:\langle e_3,e_4\rangle\rightarrow \langle e_{n-1},e_n\rangle$ and put it in the Jordan normal form. Thus, replacing $e_3$ and $e_4$ by their linear combinations and making the same linear replacement with $e_{n-1}$ and $e_n$, we may assume that either $\mu(e_2,e_3)=\gamma e_{n-1}+e_n$ and $\mu(e_2,e_4)=\gamma e_n$ for some $\gamma\in\kk$ or $\mu(e_2,e_3)=\gamma_3 e_{n-1}$ and $\mu(e_2,e_4)=\gamma_4 e_n$ for some $\gamma_3,\gamma_4\in\kk$. In the first case, replacing $e_2$ by $e_2-\gamma e_1$, one sees that $\mu\cong T^{2,2,2}({\scriptstyle \epsilon_{23}^n})$. If the second case, replacing $e_2$ by $e_2-\gamma_3e_1$, we get the algebra  $T^{2,2,2}({\scriptstyle (\gamma_4-\gamma_3)\epsilon_{24}^n})$ with the multiplication table 
$e_1e_{s+1}=e_{s+n-3}$ ($1\le s\le 3$), $e_2e_4=(\gamma_4-\gamma_3)e_n$. If $\gamma_3=\gamma_4$, then we have $\mu\cong T^{2,2,2}$. If $\gamma_3\not=\gamma_4$, then $\mu\cong T^{2,2,2}({\scriptstyle \epsilon_{24}^n})$.
\end{enumerate}

It remains to show that if $A$ is represented by $T^{2,2,2}({\scriptstyle\epsilon_{24}^n})$ or $dim\,A=7$ and $A$  is represented by $T^{2,2,2}({\scriptstyle\epsilon_{23}^4-\epsilon_{26}^7+\epsilon_{35}^7})$, then $lev(A)\le 5$. If it is not so, then $A$ has to degenerate to some algebra $B$ of level $5$ and we have $IW_1^{max}(B)\cong {\bf n}_3$, $IW_1^{max}(B)\cong T^{2,2}$ or $IW_1^{max}(B)\cong T^{2,2,2}$. Then $B$ can be represented by $\eta_5$, $T^{2,2}({\scriptstyle \epsilon_{45}^n})$, $T^{2,2,2}({\scriptstyle\epsilon_{24}^n})$ or $T^{2,2,2}({\scriptstyle\epsilon_{23}^4-\epsilon_{26}^7+\epsilon_{35}^7})$, where the last case is possible only if $dim\,A=7$.

Note that $T^{2,2,2}({\scriptstyle\epsilon_{24}^n}),T^{2,2,2}({\scriptstyle\epsilon_{23}^4-\epsilon_{26}^7+\epsilon_{35}^7})\not\to_{(1,10,n+1)}\eta_5$. If $n=7$, then we have
$T^{2,2,2}({\scriptstyle\epsilon_{23}^4-\epsilon_{26}^7+\epsilon_{35}^7})\not\to T^{2,2}({\scriptstyle \epsilon_{45}^n}),T^{2,2,2}({\scriptstyle\epsilon_{24}^n})$ by Lemma \ref{ex222}, and hence $lev\big(T^{2,2,2}({\scriptstyle\epsilon_{23}^4-\epsilon_{26}^7+\epsilon_{35}^7})\big)\le 5$. Finally, we have $T^{2,2,2}({\scriptstyle\epsilon_{24}^n})\not\to_{(1,5,n+1)} T^{2,2}({\scriptstyle \epsilon_{45}^n})$ and $T^{2,2,2}({\scriptstyle\epsilon_{24}^n})\not\to_{(1,1,n-2)} T^{2,2,2}({\scriptstyle\epsilon_{23}^4-\epsilon_{26}^7+\epsilon_{35}^7})$ that provides $lev\big(T^{2,2,2}({\scriptstyle\epsilon_{24}^n})\big)\le 5$.
\end{proof}

\begin{rema} Note that, due to the proof of Lemma \ref{T222lev}, one has $T^{2,2,2}({\scriptstyle\epsilon_{23}^4-\epsilon_{26}^7+\epsilon_{35}^7})\oplus\kk\to T^{2,2,2}({\scriptstyle\epsilon_{24}^6})\cong T^{2,2,2}({\scriptstyle\epsilon_{24}^5})\oplus\kk$ while $T^{2,2,2}({\scriptstyle\epsilon_{23}^4-\epsilon_{26}^7+\epsilon_{35}^7})\not\to T^{2,2,2}({\scriptstyle\epsilon_{24}^5})$ by Lemma \ref{ex222}. This situation is in contrast with the fact that $A\cong B$ if and only if $A\oplus\kk\cong B\oplus\kk$.
\end{rema}

\begin{lemma}\label{T2222lev}
Suppose that $IW_1^{max}(A)\cong T^{2,2,2,2}$. If $A\not\cong T^{2,2,2,2}$, then $lev(A)\ge 7$.
\end{lemma}
\begin{proof} By Lemma \ref{T2k2rest3} we have a degeneration $A\to T^{2,2,2,2}({\scriptstyle\epsilon_{23}^n})$, and hence it is enough to prove that $lev\big(T^{2,2,2,2}({\scriptstyle\epsilon_{23}^n})\big)\ge 7$. This follows from the non-trivial degeneration $T^{2,2,2,2}({\scriptstyle\epsilon_{23}^n})\xrightarrow{e_1,\dots,e_4,te_5,e_6,\dots,e_{n-4},e_n,e_{n-3},\dots,e_{n-1}}T^{2,2,2}({\scriptstyle\epsilon_{23}^{n-3}})$ (see Lemma \ref{T2k2rest1}) and that $lev\big(T^{2,2,2}({\scriptstyle\epsilon_{23}^{n-3}})\big)\ge 6$ by Lemma \ref{T222lev}.
\end{proof}

\section{Algebras with maximal IW contractions $T^{3}$ and $T^{3,2}$}

In this section we classify algebras $A$ of first five levels such that either $IW_1^{max}(A)\cong T^{3}$ and $dim\,A\ge 5$ or $IW_1^{max}(A)\cong T^{3,2}$. The remaining cases will be considered in the next section.

As usually, we start with a structure $\mu$ representing $A$ such that $\mu_{i,j}^k=0$ for $k\le max(i,j)$ and
\begin{itemize}
\item if $IW_1^{max}(A)\cong T^{3}$, then there are three integers $2\le i_1<i_2<i_3\le n$ such that $\mu(e_1,e_{i_1})=e_{i_2}$, $\mu(e_1,e_{i_2})=e_{i_3}$ and $\mu(e_1,e_s)=0$ for $s\not\in\{i_1,i_2\}$;
\item if $IW_1^{max}(A)\cong T^{3,2}$, then there are five different integers $2\le i_1<i_2<i_3\le n$ and $2\le j_1<j_2\le n$ such that $\mu(e_1,e_{i_1})=e_{i_2}$, $\mu(e_1,e_{i_2})=e_{i_3}$, $\mu(e_1,e_{j_1})=e_{j_2}$ and $\mu(e_1,e_s)=0$ for $s\not\in\{i_1,i_2,j_1\}$.
\end{itemize}

Let us set $J=\{i_1,i_2\}$ and $K=\{i_2,i_3\}$ in the case $IW_1^{max}(A)\cong T^{3}$ and $J=\{i_1,i_2,j_1\}$ and $K=\{i_2,i_3,j_2\}$ in the case $IW_1^{max}(A)\cong T^{3,2}$.
We will need the following analog of Lemma \ref{2k2}.

\begin{lemma}\label{3_32} In the settings described above, if $\mu_{i,j}^k\not=0$ for some $2\le i,j,k\le n$, then either $k\in K$ or $i,j\in J$.
\end{lemma}
\begin{proof} If $\mu_{i,j}^k\not=0$ for some $2\le i,j,k\le n$ such that $k\not\in K$ and $j\not\in J$, then one can show that for some $\alpha\in K$ the elements $L_{e_1+\alpha e_i}(e_{s})$ ($s\in J$) and $L_{e_1+\alpha e_i}(e_{j})$ are linearly independent (see the proof of Lemma \ref{2k2}). This would mean that the rank of $L_{e_1+\alpha e_i}$ is not less than $|J|+1$, and hence $IW_1^{max}(\mu)\not\to IW_{e_1+\alpha e_i}(\mu)$ that is impossible.
\end{proof}

Let now concentrate on the case $IW_1^{max}(A)\cong T^{3}$, $dim\,A\ge 5$. In this case Lemma \ref{3_32} implies that $\mu_{i,j}^{i_1}=\mu_{i,i_3}^j=0$ for any $1\le i,j\le n$. In particular, we may assume for convenience that $i_1=2$ and $i_3=n$. Thus, for some $2< r< n$, the structure $\mu$ satisfies the conditions $\mu(e_1,e_2)=e_r$, $\mu(e_1,e_r)=e_n$ and $\mu(e_1,e_s)=0$ for $s\not\in\{2,r\}$.

\begin{lemma}\label{T3rest1} If $IW_1^{max}(A)\cong T^{3}$ and $dim\,A^2>2$, then $A$ can be represented by the structure
\begin{center}
\begin{tabular}{|l|l|}
\hline
$\begin{array}{l}\eta_m({\scriptstyle\epsilon_{1,2m+1}^{n-1}+\epsilon_{2,2m+1}^n})\end{array}$&$\begin{array}{l}e_{2i-1}e_{2i}=e_{2m+1},\,\,1\le i\le m,\,\,e_1e_{2m+1}=e_{n-1},\,\,e_2e_{2m+1}=e_n\end{array}$\\
\hline
\end{tabular}
\end{center}
for some $m\ge 1$ such that $dim\,A\ge 2m+3$.
\end{lemma}
\begin{proof} Let us represent $A$ by a structure $\mu$ satisfying the conditions described above.
Due to Lemma \ref{3_32}, $dim\,A^2>2$ if and only if $\mu(e_2,e_r)$, $e_r$ and $e_n$ are linearly independent. In this case $\mu(V,V)=\langle e_r,\mu(e_2,e_r),e_n\rangle$ and we may assume that $\mu(e_2,e_r)=e_s$ for some $r<s<n$.
Replacing $e_i$ by $e_i+\mu_{2,i}^re_1$ for $3\le i\le r-1$, we may assume that $\mu_{2,i}^r=0$ for all $i\ge 3$. Note that $\mu_{i,j}^s=0$ for all $1\le i<j\le n$ except $(i,j)=(2,r)$ by Lemma \ref{3_32}. Then we have $\mu(e_2,e_i)\subset \langle e_n\rangle$ for $3\le i\le n$, $i\not=r$. Since $e_r$ and $e_s$ belong to $\Im L_{e_2}$, we have $\mu(e_2,e_i)=0$ for $i\not\in\{1,r\}$. If we interchange $e_1$ and $e_2$ and apply Lemma \ref{3_32} to the obtained algebra structure, we will see that $\mu_{i,j}^n=0$ for $1\le i<j\le n$ except $(i,j)=(1,r)$. Then we may assume that $s=n-1$ and $\mu(e_i,e_j)\subset \langle e_r\rangle$ for all $1\le i<j\le n$ except $(i,j)=(1,r)$ and $(i,j)=(2,r)$. Then the whole structure $\mu$ is determined by a skew-symmetric map $\mu:\langle e_3,\dots,e_{r-1}\rangle\times\langle e_3,\dots,e_{r-1}\rangle\rightarrow \langle e_r\rangle$. Then we can use the canonical form for a skew-symmetric bilinear map and get the required isomorphism $\mu\cong \eta_m({\scriptstyle\epsilon_{1,2m+1}^{n-1}+\epsilon_{2,2m+1}^n})$ for some $m\ge 1$.
\end{proof}

\begin{lemma}\label{T3rest2} If $IW_1^{max}(A)\cong T^{3}$ and $dim\,A^2=2$, then either $dim\,A\ge 6$ and $A$ degenerates to the algebra
\begin{center}
\begin{tabular}{|l|l|}
\hline
$\begin{array}{l}\eta_2({\scriptstyle\epsilon_{15}^n})\end{array}$&$\begin{array}{l}e_1e_2=e_3e_4=e_5,\,\,e_1e_5=e_n\end{array}$\\
\hline
\end{tabular}
\end{center}
or $A$ can be represented by a structure $\mu$ such that $\mu(e_1,e_2)=e_3$, $\mu(e_1,e_3)=e_n$, $\mu(e_1,e_i)=0$ for $i\ge 4$, $\mu(e_2,e_3)=0$ and $\mu_{i,j}^k=0$ if $k<n$ and $max(i,j)\ge 3$.
\end{lemma}
\begin{proof} As it is observed above, we may represent $A$ by a structure $\mu$ such that $\mu_{i,j}^k=0$ for $k\le max(i,j)$ and, for some $2< r< n$, one has $\mu(e_1,e_2)=e_r$, $\mu(e_1,r)=e_n$ and $\mu(e_1,e_s)=0$ for $s\not\in\{2,r\}$. Since $dim\,A^2=2$, we have $\mu_{i,j}^k=0$ whenever $k\not\in\{r,n\}$. Replacing $e_2$ by $e_2-\mu_{2,r}^ne_1$, we may assume also that $\mu(e_2,e_r)=0$.

Replacing $e_i$ by $e_i+\mu_{2,i}^re_1$ for $3\le i\le r-1$, we may assume that $\mu(e_2,e_i)\subset \langle e_n\rangle$ for all $2\le i\le n$. If $\mu_{i,j}^r=0$ for all $3\le i<j<r$, then we clearly can permutes $e_3$ and $e_r$ and get a structure representing $A$ and satisfying the required conditions. Suppose now that $\mu_{i,j}^r\not=0$ for all $3\le i<j<r$. We set $\kappa_s(t):=\begin{cases}
1,&\mbox{ if $s=1$},\\
t,&\mbox{ if $s\in\{i,j\}$},\\
t^2,&\mbox{ otherwise.}
\end{cases}$

We have
$
\mu\xrightarrow{\kappa_1(t)e_1,\dots,\kappa_n(t)e_n}T({\scriptstyle\epsilon_{1,2}^r+\epsilon_{1,r}^n+\mu_{i,j}^r\epsilon_{i,j}^r+\mu_{i,j}^n\epsilon_{i,j}^n})
$, where $T({\scriptstyle\epsilon_{1,2}^r+\epsilon_{1,r}^n+\mu_{i,j}^r\epsilon_{i,j}^r+\mu_{i,j}^n\epsilon_{i,j}^n})$ is the algebra with the multiplication table $e_1e_2=e_r$, $e_1e_r=e_n$, $e_ie_j=\mu_{i,j}^re_r+\mu_{i,j}^ne_n$. It is clear that $T({\scriptstyle\epsilon_{1,2}^r+\epsilon_{1,r}^n+\mu_{i,j}^r\epsilon_{i,j}^r+\mu_{i,j}^n\epsilon_{i,j}^n})\cong \eta_2({\scriptstyle\epsilon_{15}^n})$.
\end{proof}

To classify the algebras $A$ with $IW_1^{max}(A)\cong T^{3}$ of levels not greater than five, we will need the algebra structures presented in the next table.
\begin{center}
Table 5. {\it Algebras with $IW_1^{max}(A)\cong T^3$.}\\
\begin{tabular}{|l|l|l|}
\hline
{\rm notation}&{\rm multiplication table}&{\rm dimension}\\
\hline
$\begin{array}{l}T^3({\scriptstyle\epsilon_{23}^{n-1}})\end{array}$&$\begin{array}{l}e_1e_2=e_3,\,\,e_1e_3=e_n,\,\,e_2e_3=e_{n-1}\end{array}$&$\begin{array}{l}n\ge 5\end{array}$\\
\hline
$\begin{array}{l}T^3({\scriptstyle\epsilon_{24}^n})\end{array}$&$\begin{array}{l}e_1e_2=e_3,\,\,e_1e_3=e_2e_4=e_n\end{array}$&$\begin{array}{l}n\ge 5\end{array}$\\
\hline
$\begin{array}{l}T^3({\scriptstyle\epsilon_{34}^n})\end{array}$&$\begin{array}{l}e_1e_2=e_3,\,\,e_1e_3=e_3e_4=e_n\end{array}$&$\begin{array}{l}n\ge 5\end{array}$\\
\hline
$\begin{array}{l}T^3({\scriptstyle\epsilon_{45}^n})\end{array}$&$\begin{array}{l}e_1e_2=e_3,\,\,e_1e_3=e_4e_5=e_n\end{array}$&$\begin{array}{l}n\ge 6\end{array}$\\
\hline
\end{tabular}
\end{center}


\begin{lemma}\label{T3lev} Suppose that $IW_1^{max}(A)\cong T^3$ and $dim\,A\ge 5$. Then
\begin{itemize}
\item  $lev(A)=4$ if and only if $A$ can be represented by $T^3({\scriptstyle\epsilon_{23}^{n-1}})$ or $T^3({\scriptstyle\epsilon_{24}^{n}})$;
\item  $lev(A)=5$ if and only if either $A$ can be represented by $T^3({\scriptstyle\epsilon_{34}^n})$ or $dim\,A=6$ and $A$ can be represented by $T^3({\scriptstyle\epsilon_{45}^{6}})$.
\end{itemize}
\end{lemma}
\begin{proof} Note first that we have non-trivial degenerations $T^3({\scriptstyle\epsilon_{23}^{n-1}}),T^3({\scriptstyle\epsilon_{24}^{n}})\to T^3$, $T^3({\scriptstyle\epsilon_{34}^n})\xrightarrow{te_1,e_2+e_3,te_3+te_n,t^2e_4,e_5\dots,e_{n-1},t^2e_n} T^3({\scriptstyle\epsilon_{24}^{n}})$ and $T^3({\scriptstyle\epsilon_{45}^n})\xrightarrow{e_1,e_2-e_5,e_3,e_4,te_5,e_6\dots,e_n} T^3({\scriptstyle\epsilon_{24}^{n}})$, and hence $lev\big(T^3({\scriptstyle\epsilon_{23}^{n-1}})\big)\ge 4$, $lev\big(T^3({\scriptstyle\epsilon_{24}^{n}})\big)\ge 4$, $lev\big(T^3({\scriptstyle\epsilon_{34}^{n}})\big)\ge 5$ and $lev\big(T^3({\scriptstyle\epsilon_{45}^{n}})\big)\ge 5$.

Note that $\eta_m({\scriptstyle\epsilon_{1,2m+1}^{n-1}+\epsilon_{2,2m+1}^n})\to \eta_2({\scriptstyle\epsilon_{15}^{n-1}+\epsilon_{25}^n})\xrightarrow{e_1,\dots,e_{n-2},\frac{1}{t}e_n,e_{n-1}} \eta_2({\scriptstyle\epsilon_{15}^n})$ for $m\ge 2$ and
$\eta_2({\scriptstyle\epsilon_{15}^n})\xrightarrow{te_1,e_2-e_5,te_5-te_n,te_3,te_4,e_6,\dots,e_{n-1},t^2e_n}T^3({\scriptstyle\epsilon_{45}^n})$. Due to Lemmas \ref{T3rest1} and \ref{T3rest2}, if $A$ has level not greater than five, then it can be represented either by the structure $\eta_1({\scriptstyle\epsilon_{13}^{n-1}+\epsilon_{23}^n})\cong T^3({\scriptstyle\epsilon_{23}^{n-1}})$ or  by a structure $\mu$ such that $\mu(e_1,e_2)=e_3$, $\mu(e_1,e_3)=e_n$, $\mu(e_1,e_i)=0$ for $i\ge 4$, $\mu(e_2,e_3)=0$ and $\mu_{i,j}^k=0$ if $k<n$ and $max(i,j)\ge 3$.
It remains to consider the last case to prove that there are no algebras $A$ of levels four and five with $IW_1^{max}(A)\cong T^3$ except the algebras mentioned in the statement of this lemma. Let us consider three cases.
\begin{enumerate}
\item $\mu(e_i,e_j)=0$ for all $3\le i,j\le n-1$. In this case the multiplication table of $\mu$ is defined by the equalities $e_1e_2=e_3$, $e_1e_3=e_n$, $e_2e_i=\alpha_ie_n$ ($4\le i\le n-1$), where $\alpha_i$ are some elements of the field $\kk$. It is clear that if all $\alpha_i$ are zero, then $\mu\cong T^3$ and if minimum one of the elements $\alpha_i$ is not zero, then $\mu\cong T^3({\scriptstyle\epsilon_{24}^n})$.

\item $\mu(e_i,e_j)\not=0$ for some $4\le i<j\le n-1$. We may assume in this case that $\mu(e_i,e_j)=e_n$. We set $\kappa_s(t):=\begin{cases}
1,&\mbox{ if $s=1$},\\
t,&\mbox{ if $s\in\{i,j\}$},\\
t^2,&\mbox{ otherwise.}
\end{cases}$

Then we have the degeneration
$
\mu\xrightarrow{\kappa_1(t)e_1,\dots,\kappa_n(t)e_n}T^3({\scriptstyle\epsilon_{i,j}^{n}})
$, where $T^3({\scriptstyle\epsilon_{i,j}^{n}})$ is the algebra with the multiplication table $e_1e_2=e_3$, $e_1e_3=e_ie_j=e_n$. It is clear that $T^3({\scriptstyle\epsilon_{i,j}^{n}})\cong T^3({\scriptstyle\epsilon_{45}^{n}})$.

\item $\mu(e_3,e_i)\not=0$ for some $4\le i\le n-1$. We may assume in this case that $\mu(e_3,e_i)=e_n$.
Then we have the degeneration
$
\mu\xrightarrow{\frac{1}{t}e_1,t^2e_2,te_3,\frac{1}{t}e_i,t^2e_4\dots,t^2e_{i-1},t^2e_{i+1}\dots,t^2e_{n-1},e_n}T^3({\scriptstyle\epsilon_{34}^{n}})$.
\end{enumerate}
It remains to note that if $n\ge 7$, then $T^3({\scriptstyle\epsilon_{45}^{n}})\xrightarrow{e_1,te_2,e_3-\frac{1}{t}e_{n-1},e_4,\dots,e_n}T^{2,2}({\scriptstyle \epsilon_{45}^n})$.

It remains to show that $lev\big(T^3({\scriptstyle\epsilon_{23}^{n-1}})\big)\le 4$, $lev\big(T^3({\scriptstyle\epsilon_{24}^{n}})\big)\le 4$, $lev\big(T^3({\scriptstyle\epsilon_{34}^{n}})\big)\le 5$ and, in the case $n=6$, $lev\big(T^3({\scriptstyle\epsilon_{45}^{6}})\big)\le 5$.

Suppose that $A$ is represented by $T^3({\scriptstyle\epsilon_{23}^{n-1}})$ or $T^3({\scriptstyle\epsilon_{24}^{n}})$. If $lev(A)>4$, then $A$ degenerates to some algebra $B$  of level $4$ and we have $IW_1^{max}(B)\cong {\bf n}_3$, $IW_1^{max}(B)\cong T^{2,2}$ or $IW_1^{max}(B)\cong T^3$. Then $B$ can be represented by $\eta_4$, $T^{2,2}({\scriptstyle \epsilon_{34}^n})$, $T^3({\scriptstyle\epsilon_{23}^{n-1}})$ or $T^3({\scriptstyle\epsilon_{24}^{n}})$.
Note that $T^3({\scriptstyle\epsilon_{23}^{n-1}})\not\to_{(1,4,n+1)}\eta_4,T^{2,2}({\scriptstyle \epsilon_{34}^n}),T^3({\scriptstyle\epsilon_{24}^{n}})$; $T^3({\scriptstyle\epsilon_{24}^{n}})\not\to_{(1,8,n+1)}\eta_4$;
$T^3({\scriptstyle\epsilon_{24}^{n}})\not\to_{(1,3,n),(3,3,n+1)}T^{2,2}({\scriptstyle \epsilon_{34}^n})$ and $T^3({\scriptstyle\epsilon_{24}^{n}})\not\to_{(1,1,n-1)}T^3({\scriptstyle\epsilon_{23}^{n-1}})$. Thus, $lev\big(T^3({\scriptstyle\epsilon_{23}^{n-1}})\big)\le 4$ and $lev\big(T^3({\scriptstyle\epsilon_{24}^{n}})\big)\le 4$.

Suppose that either $A$ is represented by $T^3({\scriptstyle\epsilon_{34}^{n}})$ or $dim\,A=6$ and $A$ is represented by $T^3({\scriptstyle\epsilon_{45}^{n}})$. If $lev(A)>5$, then $A$ degenerates to some algebra $B$  of level $5$ and we again have $IW_1^{max}(B)\cong {\bf n}_3$, $IW_1^{max}(B)\cong T^{2,2}$ or $IW_1^{max}(B)\cong T^3$. Then $B$ can be represented by $\eta_5$, $T^{2,2}({\scriptstyle \epsilon_{45}^n})$, $T^3({\scriptstyle\epsilon_{34}^{n}})$ or $T^3({\scriptstyle\epsilon_{45}^{n}})$, where the case $T^{2,2}({\scriptstyle \epsilon_{45}^n})$ is possible only if $dim\,A\ge 7$.
Note that $T^3({\scriptstyle\epsilon_{34}^{n}})\not\to_{(1,5,n+1)}\eta_5,T^{2,2}({\scriptstyle \epsilon_{45}^n}),T^3({\scriptstyle\epsilon_{45}^{n}})$; $T^3({\scriptstyle\epsilon_{45}^{n}})\not\to_{(1,10,n+1)}\eta_5$
 and $T^3({\scriptstyle\epsilon_{45}^{n}})\not\to_{(1,1,n-1),(1,3,n),(2,n-1,n+1)}T^3({\scriptstyle\epsilon_{34}^{n}})$. Thus, $lev\big(T^3({\scriptstyle\epsilon_{34}^{n}})\big)\le 5$ and, in the case $n=6$, $lev\big(T^3({\scriptstyle\epsilon_{45}^{n}})\big)\le 5$.
\end{proof}

Let us now consider the case $IW_1^{max}(A)\cong T^{3,2}$. For this piece of our classification, we will need one more algebra structure.
\begin{center}
\begin{tabular}{|l|l|l|}
\hline
{\rm notation}&{\rm multiplication table}&{\rm dimension}\\
\hline
$\begin{array}{l}T^{3,2}({\scriptstyle\epsilon_{23}^{n}})\end{array}$&$\begin{array}{l}e_1e_2=e_{n-1},\,\,e_1e_3=e_4,\,\,e_1e_4=e_n,\,\,e_2e_3=e_n\end{array}$&$\begin{array}{l}n\ge 6\end{array}$\\
\hline
\end{tabular}
\end{center}

\begin{lemma}\label{T32lev} Suppose that $IW_1^{max}(A)\cong T^{3,2}$. Then $lev(A)=5$ if and only if $A$ can be represented by $T^{3,2}({\scriptstyle\epsilon_{23}^n})$.
\end{lemma}
\begin{proof} We want to show first that if $IW_1^{max}(A)\cong T^{3,2}$ and $A\not\cong T^{3,2}$, then $A\to T^{3,2}({\scriptstyle\epsilon_{23}^n})$.
We represent the algebra $A$ by a structure $\mu$ such that $\mu_{i,j}^k=0$ for $k\le max(i,j)$ and there are five different integers $2\le i_1<i_2<i_3\le n$ and $2\le j_1<j_2\le n$ such that $\mu(e_1,e_{i_1})=e_{i_2}$, $\mu(e_1,e_{i_2})=e_{i_3}$, $\mu(e_1,e_{j_1})=e_{j_2}$ and $\mu(e_1,e_s)=0$ for $s\not\in\{i_1,i_2,j_1\}$. Due to Lemma \ref{3_32}, we also have $\mu_{i,j}^k=0$ if $k\not\in \{i_2,i_3,j_2\}$ and $j\not\in\{i_1,i_2,j_1\}$. In particular, $\mu_{i,j}^{i_1}=0$ for any $1\le i,j\le n$.

Note first that, for $2\le i\le n$ and any $\alpha\in\kk$, one has
$$0=(L_{e_1+\alpha e_{i}})^3(e_{i_1})=\alpha\big(\mu(e_i,e_{i_3})+L_{e_1}\mu(e_i,e_{i_2})+(L_{e_1})^2\mu(e_{i},e_{i_1})\big)+\alpha^2v,$$
where $v\in V$ does not depend on $\alpha$. Note that $\mu_{i,i_1}^{i_1}=0$, and hence $(L_{e_1})^2\mu(e_{i},e_{i_1})=0$. Since $\mu_{i,i_2}^{i_1}=\mu_{i,i_2}^{i_2}=0$, we get
$\mu(e_{i_3},e_i)=\mu_{i,i_2}^{j_1}e_{j_2}$, in particular, $\mu(e_i,e_{i_3})=0$ for $i\not=i_1$.

If $\mu_{i_1,i_2}^{j_1}\not=0$, then it is enough to prove that $\chi\to T^{3,2}({\scriptstyle\epsilon_{23}^n})$ for the structure $\chi$ defined by the degeneration $\mu\xrightarrow{\kappa_1(t)e_1,\dots,\kappa_n(t)e_n}\chi$, where $\kappa_s(t):=\begin{cases}
1,&\mbox{ if $s=1$},\\
t^2,&\mbox{ if $s\in\{i_1,i_2,i_3\}$},\\
t^4,&\mbox{ if $s\in\{j_1,j_2\}$},\\
t^3,&\mbox{ otherwise.}
\end{cases}$

The algebra $\chi$ has multiplication table
$$e_1e_{i_1}=e_{i_2},\,\,e_1e_{i_2}=e_{i_3},\,\,e_1e_{j_1}=e_{j_2},\,\,e_{i_1}e_{i_2}=\mu_{i_1,i_2}^{j_1}e_{j_1}+\mu_{i_1,i_2}^{j_2}e_{j_2},\,\,e_{i_1}e_{i_3}=-\mu_{i_1,i_2}^{j_1}e_{j_2}.$$
Replacing $\chi$ by an isomorphic structure, we may assume that $i_1=2$, $i_2=3$, $i_3=n-1$, $j_1=4$, $j_2=n$ and $\mu_{i_1,i_2}^{j_1}=1$, $\mu_{i_1,i_2}^{j_2}=0$. Then we have
$$\chi\xrightarrow{e_1,te_2,e_3-\frac{1}{t}e_4+\frac{1}{t^2}e_n,e_4-\frac{1}{t}e_n,e_5,\dots,e_{n-2},e_{n-1}-\frac{1}{t}e_n, e_n}T^{3,2}({\scriptstyle\epsilon_{23}^n}).$$
Note also that, for $2\le i\le n$ and any $\alpha\in\kk$, one has
$$0=(L_{e_1+\alpha e_i})^3(e_{j_1})=\alpha\big((L_{e_1})^2\mu(e_i,e_{j_1})+L_{e_1}\mu(e_i,e_{j_2})\big)+\alpha^2v,$$
where $v\in V$ does not depend on $\alpha$. It follows from $\mu_{i,j_1}^{i_1}=0$ that $\mu_{i,j_2}^{i_2}=0$.

Now we may assume that $\mu_{i_1,i_2}^{j_1}=0$, and hence $\mu(e_{i_1},e_{i_3})=0$ by the argument above. Then without loss of generality we may assume that $i_1=3$, $i_3=n$, $j_1=2$ and $j_2=n-1$. From here on we denote also $i_2$ by $r$.
Replacing $e_i$ by $e_i+\mu_{3,i}^re_1$ for $2\le i\le r-1$, we may assume that $\mu_{3,i}^r=0$ for all $2\le i\le n$. Suppose that there are $2\le i<j\le r-1$ such that $\mu_{i,j}^r\not=0$. If $\mu_{2,j}^r=0$, the we can replace $e_2$ by $e_2+e_i$ and assume that $\mu_{2,j}^r\not=0$ for some $4\le j\le r-1$.

It is enough to prove that $\chi\to T^{3,2}({\scriptstyle\epsilon_{23}^n})$ for the structure $\chi$ defined by the degeneration $\mu\xrightarrow{\kappa_1(t)e_1,\dots,\kappa_n(t)e_n}\chi$, where $\kappa_s(t):=\begin{cases}
1,&\mbox{ if $s=1$},\\
t^2,&\mbox{ if $s\in\{2,j,n-1\}$},\\
t^4,&\mbox{ if $s\in\{3,r,n\}$},\\
t^3,&\mbox{ otherwise.}
\end{cases}$

The algebra $\chi$ has multiplication table
\begin{multline*}e_1e_2=e_{n-1},\,\,e_1e_3=e_r,\,\,e_1e_r=e_n,\,\,e_2e_j=\mu_{2,j}^re_r+\mu_{2,j}^ne_n,\\e_2e_{n-1}=\mu_{2,n-1}^ne_n,\,\,e_je_{n-1}=\mu_{j,n-1}^ne_n.
\end{multline*}
Replacing $\chi$ by an isomorphic structure, we may assume that $j=4$, $r=5$, $\mu_{2,j}^r=1$ and $\mu_{2,j}^n=0$. Then we have
$$\chi\xrightarrow{e_1,te_2,e_3+e_4-\frac{1}{t}e_5,e_5-\frac{1}{t}e_n,te_4,e_6,\dots,e_{n-2},te_{n-1}, e_n}T^{3,2}({\scriptstyle\epsilon_{23}^n}).$$

Hence, we may assume that $\mu_{i,j}^r=0$ for $1\le i<j\le n$, $(i,j)\not=(1,3)$. Then we may assume also that $r=4$. Suppose now that $dim\,A^2>3$. Then there is some $5\le i\le n-3$ such that one of the elements $\mu_{2,3}^i$, $\mu_{2,4}^i$ and $\mu_{3,4}^i$ is nonzero. Since we can replace $e_3$ and $e_4$ by $e_3+e_4$ and $e_4+e_n$ or $e_2$ and $e_{n-1}$ by $e_2+e_4$ and $e_{n-1}+e_n$, we may assume that $\mu_{23}^i\not=0$.

It is enough to prove that $\chi\to T^{3,2}({\scriptstyle\epsilon_{23}^n})$ for the structure $\chi$ defined by the degeneration $\mu\xrightarrow{\kappa_1(t)e_1,\dots,\kappa_n(t)e_n}\chi$, where $\kappa_s(t):=\begin{cases}
t,&\mbox{ if $s=1$},\\
t^2,&\mbox{ if $s\in\{2,3\}$},\\
t^4,&\mbox{ if $s\in\{i,n\}$},\\
t^3,&\mbox{ otherwise.}
\end{cases}$

The algebra $\chi$ has multiplication table
$e_1e_2=e_{n-1},\,\,e_1e_3=e_4,\,\,e_1e_4=e_n,\,\,e_2e_3=\mu_{23}^ie_i+\mu_{23}^ne_n.$
Replacing $\chi$ by an isomorphic structure, we may assume that $\mu_{23}^i=\mu_{23}^n=1$. Then we have
$$\chi\xrightarrow{e_1,\dots,e_{i-1},\frac{1}{t}e_i,e_{i+1},\dots, e_n}T^{3,2}({\scriptstyle\epsilon_{23}^n}).$$

From here on we may assume that $\mu(V,V)=\{e_4,e_{n-1},e_n\}$.
Note that if $\mu_{23}^n\not=0$, then
$$\chi\xrightarrow{te_1,t^2e_2,t^2e_3,t^3e_4,\dots,t^3e_{n-1}, \mu_{23}^nt^4e_n}T^{3,2}({\scriptstyle\epsilon_{23}^n}).$$
If $\mu_{23}^{n-1}+\mu_{34}^n\not=0$ or $\mu_{34}^{n-1}\not=0$, then for some $\alpha\in\kk$ one has $\mu(e_2+\alpha e_4,e_3)\not\in\langle e _{n-1}+\alpha e_n\rangle$, i.e. we can replace $e_2$ and $e_{n-1}$ by $e_2+\alpha e_4$ and $e _{n-1}+\alpha e_n$ that returns us to the case $\mu_{23}^n\not=0$. Suppose that $\mu_{23}^n=\mu_{23}^{n-1}+\mu_{34}^n=\mu_{34}^{n-1}=0$. Replacung $e_3$ by $e_3-\mu_{34}^ne_1$, we may assume that $\mu(e_2,e_3)=\mu(e_3,e_4)=0$.
If $\mu(e_2,e_4)\not=0$, then replacing $e_3$ and $e_4$ by $e_3+e_4$ and $e_4+e_n$, we may assume that $\mu(e_2,e_3)\not=0$ and $\mu(e_3,e_4)=0$ that comes down to the case $\mu_{23}^n\not=0$ as was explained above.
Analogously, if $\mu(e_3,e_i)$ or $\mu(e_4,e_i)$ is nonzero for some $i\ge 5$, then adding $e_i$ to $e_2$ we can reduce everything to the case $\mu_{23}^n\not=0$. If $\mu(e_3,e_i)=\mu(e_4,e_i)=0$ and $\mu(e_2,e_i)$ is nonzero for some $i\ge 5$, then we are done by the replacement of $e_3$ by $e_3+e_i$. Finally, if $\mu(e_i,e_j)\not=0$ for two integers $i,j\ge 5$, then we can add $e_j$ to $e_2$ and return to the case $\mu(e_2,e_i)\not=0$. If $\mu(e_i,e_j)=0$ for all $2\le i,j\le n$, then $A$ is represented by $T^{3,2}$ that contradicts our assumptions.

Note that $T^{3,2}({\scriptstyle\epsilon_{23}^n})\xrightarrow{\not\cong}T^{3,2}$, and hence $lev\big(T^{3,2}({\scriptstyle\epsilon_{23}^n}))\big)\ge 5$ and algebra $A$ with $IW_1^{max}(A)\cong T^{3,2}$ can have level five only if $A$ is represented by $T^{3,2}({\scriptstyle\epsilon_{23}^n})$. If $lev\big(T^{3,2}({\scriptstyle\epsilon_{23}^n}))\big)> 5$, then $T^{3,2}({\scriptstyle\epsilon_{23}^n})$ degenerates to some algebra $B$ of level five. Since $IW_1^{max}(A)\to IW_1^{max}(B)$, we have $IW_1^{max}(B)\cong {\bf n}_3$, $IW_1^{max}(B)\cong T^{2,2}$, $IW_1^{max}(B)\cong T^3$, $IW_1^{max}(B)\cong T^{2,2,2}$ or $IW_1^{max}(B)\cong T^{3,2}$. Then $B$ can be represented by $\eta_5$, $T^{2,2}({\scriptstyle\epsilon_{45}^{n}})$, $T^3({\scriptstyle\epsilon_{34}^n})$, $T^3({\scriptstyle\epsilon_{45}^{n}})$, $T^{2,2,2}({\scriptstyle\epsilon_{24}^n})$, $T^{2,2,2}({\scriptstyle\epsilon_{23}^4-\epsilon_{26}^7+\epsilon_{35}^7})$ or $T^{3,2}({\scriptstyle\epsilon_{23}^n})$.
In the last case $B\cong A$ by our arguments.
We have also $$T^{3,2}({\scriptstyle\epsilon_{23}^n})\not\to_{(1,5,n+1)} \eta_5, T^{2,2}({\scriptstyle\epsilon_{45}^{n}}), T^3({\scriptstyle\epsilon_{45}^{n}}), T^{2,2,2}({\scriptstyle\epsilon_{23}^4-\epsilon_{26}^7+\epsilon_{35}^7});$$ 
$T^{3,2}({\scriptstyle\epsilon_{23}^n})\not\to T^3({\scriptstyle\epsilon_{34}^n})$ because $T^{3,2}({\scriptstyle\epsilon_{23}^n})$ is a Lie algebra and $T^3({\scriptstyle\epsilon_{34}^n})$ is not, and finally $T^{3,2}({\scriptstyle\epsilon_{23}^n})\not\to_{(1,4,n),(2,4,n+1),(2,3,n)} T^{2,2,2}({\scriptstyle\epsilon_{24}^n})$.
\end{proof}

\section{Algebras with maximal IW contractions of non-stable level}

In this section we consider algebras $A$ such that $lev\big(IW_1^{max}(A)\big)<lev_\infty\big(IW_1^{max}(A)\big)$. Due to Table 1, this occurs when either $IW_1^{max}(A)\cong T^3$ and $dim\,A=4$ or $IW_1^{max}(A)\cong T^4$ and $dim\,A=5$.

\begin{lemma}\label{T3d4} If $IW_1^{max}(A)\cong T^3$ and $dim\,A=4$, then $A$ can be represented by $T^3$.
\end{lemma}
\begin{proof} Since in the considered case $A$ is nilpotent, we may assume that $A$ is represented by a structure $\mu$ such that $\mu(e_1,e_2)=e_3$, $\mu(e_1,e_3)=e_4$ and $\mu_{i,j}^k=0$ if $k\le max(i,j)$. Then the only nonzero product except ones  we have already mentioned is $\mu(e_2,e_3)=\mu_{23}^4e_4$. Then in the basis $e_1,e_2-\mu_{23}^4e_1,e_3,e_4$ the structure $\mu$ has the same structure constants as $T^3$, i.e. $A$ can be represented by $T^3$.
\end{proof}

It remains to study the case of five-dimensional algebra $A$ with $IW_1^{max}(A)\cong T^4$.
The main difficulty of this case is that we do not have nilpotence of the algebra $A$ automatically.
Let us now introduce the five-dimensional algebra $T^4({\scriptstyle\epsilon_{23}^5})$ with $IW_1^{max}\big(T({\scriptstyle\epsilon_{23}^5})\big)\cong T^4$.

\begin{center}
\begin{tabular}{|l|l|l|}
\hline
{\rm notation}&{\rm multiplication table}&{\rm dimension}\\
\hline
$\begin{array}{l}T^4({\scriptstyle\epsilon_{23}^5})\end{array}$&$\begin{array}{l}e_1e_i=e_{i+1},\,\,i=2,3,4,\,\,e_2e_3=e_5\end{array}$&$\begin{array}{l}n=5\end{array}$\\
\hline
\end{tabular}
\end{center}

\begin{lemma}\label{T4lev} Suppose that $IW_1^{max}(A)\cong T^4$ and $dim\,A=5$. Then $A$ has level five if and only if it can be represented by $T^4({\scriptstyle\epsilon_{23}^5})$.
\end{lemma}
\begin{proof} Let us show first that if $IW_1^{max}(A)\cong T^4$, $dim\,A=5$ and $A\not\cong T^4$, then  $A\to T^4({\scriptstyle\epsilon_{23}^5})$.
We may assume that $A$ is represented by a structure $\mu$ such that $IW_{e_1}(\mu)=T^4$. This means that $\mu(e_1,e_i)=\mu_{1,i}^1e_1+e_{i+1}$ for $2\le i\le 4$ and $\mu(e_1,e_5)=\mu_{15}^1e_1$. We have $\mu_{15}^1=0$ by the nilpotence of the operator $L_{e_5}$. Replacing $e_i$ by $\mu_{1,i-1}^1e_1+e_i$ for $3\le i\le 5$, we may assume that $\mu_{1,i}^1=0$ for all $2\le i\le 5$.

Let us pick some $\alpha,\beta\in\kk$ and consider a new basis of $V$ defined by the equalities $f_1=e_1$, $f_2=e_2+\alpha e_3+\beta e_4$, $f_3=e_3+\alpha e_4+\beta e_5$, $e_4=e_4+\alpha e_5$, $f_5=e_5$. Note that for any choice of $\alpha$ and $\beta$ we have $\mu(f_1,f_i)=f_{i+1}$ for $2\le i\le 4$ and $\mu(f_1,f_5)=0$. We will consider two cases.

\begin{itemize}
\item Suppose that $\mu(f_2,f_3)\subset \langle f_4\rangle=\langle e_4+\alpha e_5\rangle$ for any choice of $\alpha,\beta\in\kk$. Note that
\begin{multline*}
\mu(f_2,f_3)=\mu(e_2,e_3)+\alpha\mu(e_2,e_4)+\alpha^2\mu(e_3,e_4)\\
+\beta \big(\mu(e_2,e_5)-\mu(e_3,e_4)\big)+\alpha\beta\mu(e_3,e_5)+\beta^2 \mu(e_4,e_5).
\end{multline*}
Now one sees that the condition $\mu(f_2,f_3)\subset \langle f_4\rangle$ is equivalent to the qualities 
\begin{multline*}
\mu(e_2,e_3)=\mu_{23}^4e_4,\,\,\mu(e_2,e_4)=\mu_{24}^4e_4+\mu_{24}^5e_5,\,\,\mu(e_3,e_4)=\mu_{34}^5e_5,\\
\mu(e_2,e_5)=\mu_{25}^4e_4+\mu_{25}^5e_5,\,\,\mu(e_3,e_5)=\mu_{35}^5e_5,\,\,\mu(e_4,e_5)=0
\end{multline*}
with $\mu_{24}^5=\mu_{23}^4$, $\mu_{25}^5=\mu_{34}^5=\mu_{24}^4$ and $\mu_{35}^5=\mu_{25}^4$. The nilpotence of $L_{e_3}$ implies $\mu_{35}^5=0$, and hence $\mu(e_2,e_5)=\mu_{25}^5e_5$. Then the nilpotence of $L_{e_2}$ implies $\mu_{25}^5=0$, and hence $\mu$ is a structure whose nonzero products of basic elements are $\mu(e_1,e_i)=e_{i+1}$ for $2\le i\le 4$, $\mu(e_2,e_3)=\gamma e_4$ and $\mu(e_2,e_4)=\gamma e_5$ for some $\gamma\in \kk$. Considering the basis $e_1$, $e_2-\gamma e_1$, $e_3$, $e_4$, $e_5$, one sees that $\mu\cong T^4$.

\item Suppose that $\mu(f_2,f_3)\not\subset \langle f_4\rangle$ for some choice of $\alpha,\beta\in\kk$. Then we may assume that $\mu(e_2,e_3)\not\in\langle e_4\rangle$. Suppose that $\mu_{23}^5=0$. Let us denote by $v$ the vector $\mu_{23}^1e_1+\mu_{23}^2e_2+\mu_{23}^3e_3$ which is nonzero by our assumption.
Then one has
$$L_{e_2-\mu_{23}^4e_1}(e_3)=v,\,\,L_{e_2-\mu_{23}^4e_1}(v)=\mu_{23}^3 v-(\mu_{23}^1+\mu_{23}^2\mu_{23}^4)e_3.$$
The nilpotence of $L_{e_2-\mu_{23}^4e_1}$ implies $\mu_{23}^3=\mu_{23}^1+\mu_{23}^2\mu_{23}^4=0$. Now we have
$L_{e_3}(v)=-\mu_{23}^2v$ and the nilpotence of $L_{e_3}$ implies $\mu_{23}^2=0$, and hence $\mu_{23}^1=0$ that contradicts the fact that $v\not=0$.
Thus, we have $\mu_{23}^5\not=0$. Then $\mu\xrightarrow{te_1,\frac{t^2}{\mu_{23}^5}e_2,\frac{t^3}{\mu_{23}^5}e_3,\frac{t^4}{\mu_{23}^5}e_4,\frac{t^5}{\mu_{23}^5}e_5}T^4({\scriptstyle\epsilon_{23}^5})$.
\end{itemize}

Note that $T^4({\scriptstyle\epsilon_{23}^5})\xrightarrow{\not\cong}T^4$, and hence $lev\big(T^4({\scriptstyle\epsilon_{23}^5})\big)\ge 5$ and algebra $A$ with $IW_1^{max}(A)\cong T^4$ can have level five only if $A$ is represented by $T^4({\scriptstyle\epsilon_{23}^5})$. If $lev\big(T^4({\scriptstyle\epsilon_{23}^5})\big)> 5$, then $T^4({\scriptstyle\epsilon_{23}^5})$ degenerates to some algebra $B$ of level five. Since $IW_1^{max}(A)\to IW_1^{max}(B)$, we have $IW_1^{max}(B)\cong {\bf n}_3$, $IW_1^{max}(B)\cong T^{2,2}$, $IW_1^{max}(B)\cong T^3$ or $IW_1^{max}(B)\cong T^4$. Then $B$ can be represented by $T^3({\scriptstyle\epsilon_{34}^5})$ or $T^4({\scriptstyle\epsilon_{23}^5})$ because all other algebras with these maximal one-dimensional IW contractions and level five have dimensions greater than five.
In the last case $B\cong A$ by our arguments. Finally, $T^4({\scriptstyle\epsilon_{23}^5})\not\to_{(1,5,6),(1,4,5),(1,1,3),(2,4,6),(2,3,5)} T^3({\scriptstyle\epsilon_{34}^5})$.
\end{proof}


\section{Main Theorem}

In this section we state and prove the theorems giving the classification of anticommutative Engel algebras of levels from three to five, which is the main aim of this paper. We will give  also some consequences of our classification.
Let us recall that the finite dimensional algebras $A$ and $B$ are {\it stably isomorphic} if $A\oplus\kk^k\cong B\oplus\kk^l$ for some integers $k,l$.

\begin{theorem}\label{lev3}
Let $A$ be an $n$-dimensional anticommutative Engel algebra.
\begin{itemize}
\item If $dim\,A\le 4$, then $A$ has level not greater than two.
\item If $dim\,A=5$, then $A$ has level $3$ if and only if it can be represented by $T^3$.
\item If $dim\,A=6$, then $A$ has level $3$ if and only if it can be represented by $T^{2,2}({\scriptstyle\epsilon_{23}^{4}})$, $T^{2,2}({\scriptstyle\epsilon_{24}^6})$ or $T^3$.
\item If $dim\,A\ge 7$, then $A$ has level $3$ if and only if it can be represented by $\eta_3$, $T^{2,2}({\scriptstyle\epsilon_{23}^{n-2}})$, $T^{2,2}({\scriptstyle\epsilon_{24}^n})$, $T^3$ or $T^{2,2,2}$.
\end{itemize}
\end{theorem}
\begin{proof} The algebra $A$ can have level $3$ only if $lev\big(IW_1^{max}(A)\big)\le 3$. Then $IW_1^{max}(A)$ is isomorphic to one of the structures ${\bf n}_3$, $T^{2,2}$, $T^3$ and $T^{2,2,2}$. In the first case $A$ has level $3$ if and only if it is represented by $\eta_3$. If either $IW_1^{max}(A)\cong T^{2,2,2}$ or $dim\,A\ge 5$ and $IW_1^{max}(A)\cong T^3$, then $A$ has level $3$ if and only if $A\cong IW_1^{max}(A)$.
The remaining part of the theorem follows from Lemmas \ref{T22lev} and \ref{T3d4}.
\end{proof}

Now we can recover the nilpotent part of the classification from \cite{gorb93} in a correct form.

\begin{coro} An anticommutative Engel algebra $A$ has infinite level $3$ if and only if it is stably isomorphic to $\eta_3$, $T^{2,2}({\scriptstyle\epsilon_{23}^{n-2}})$, $T^{2,2}({\scriptstyle\epsilon_{24}^n})$, $T^3$ or $T^{2,2,2}$.
\end{coro}

\begin{theorem}\label{lev4}
Let $A$ be an $n$-dimensional anticommutative Engel algebra.
\begin{itemize}
\item If $dim\,A=5$, then $A$ has level $4$ if and only if it can be represented by $T^3({\scriptstyle\epsilon_{23}^{4}})$, $T^3({\scriptstyle\epsilon_{24}^{5}})$ or $T^4$.
\item If $dim\,A=6$, then $A$ has level $4$ if and only if it can be represented by $T^{2,2}({\scriptstyle\epsilon_{34}^{6}})$, $T^3({\scriptstyle\epsilon_{23}^{5}})$, $T^3({\scriptstyle\epsilon_{24}^{6}})$ or $T^{3,2}$.
\item If $dim\,A=7,8$, then $A$ has level $4$ if and only if it can be represented by $T^{2,2}({\scriptstyle\epsilon_{34}^{n}})$, $T^3({\scriptstyle\epsilon_{23}^{n-1}})$, $T^3({\scriptstyle\epsilon_{24}^{n}})$, $T^{2,2,2}({\scriptstyle\epsilon_{23}^{n}})$ or $T^{3,2}$.
\item If $dim\,A\ge 9$, then $A$ has level $4$ if and only if it can be represented by $\eta_4$, $T^{2,2}({\scriptstyle\epsilon_{34}^{n}})$, $T^3({\scriptstyle\epsilon_{23}^{n-1}})$, $T^3({\scriptstyle\epsilon_{24}^{n}})$, $T^{2,2,2}({\scriptstyle\epsilon_{23}^{n}})$, $T^{3,2}$ or $T^{2,2,2,2}$.
\end{itemize}
\end{theorem}
\begin{proof} The algebra $A$ can have level $4$ only if $lev\big(IW_1^{max}(A)\big)\le 4$. Then $IW_1^{max}(A)$ is isomorphic to one of the structures ${\bf n}_3$, $T^{2,2}$, $T^3$, $T^{2,2,2}$, $T^{3,2}$, $T^{2,2,2,2}$ or $T^4$, where the last case can occur only if $dim\,A=5$. In the first case $A$ has level $4$ if and only if it is represented by $\eta_4$. If $IW_1^{max}(A)\cong T^{3,2}$, $IW_1^{max}(A)\cong T^{2,2,2,2}$ or $IW_1^{max}(A)\cong T^4$, where in the last case $dim\,A=5$, then $A$ has level $4$ if and only if $A\cong IW_1^{max}(A)$.
The remaining part of the theorem follows from Lemmas \ref{T22lev}, \ref{T222lev}, \ref{T3lev} and \ref{T3d4}.
\end{proof}

\begin{coro} An anticommutative Engel algebra $A$ has infinite level $4$ if and only if it is stably isomorphic to $\eta_4$, $T^{2,2}({\scriptstyle\epsilon_{34}^{n}})$, $T^3({\scriptstyle\epsilon_{23}^{n-1}})$, $T^3({\scriptstyle\epsilon_{24}^{n}})$, $T^{2,2,2}({\scriptstyle\epsilon_{23}^{n}})$, $T^{3,2}$ or $T^{2,2,2,2}$.
\end{coro}

\begin{coro} Any anticommutative Engel algebra of level not greater than $4$ is a Lie algebra.
\end{coro}

\begin{theorem}\label{lev5}
Let $A$ be an $n$-dimensional anticommutative Engel algebra.
\begin{itemize}
\item If $dim\,A=5$, then $A$ has level $5$ if and only if it can be represented by $T^3({\scriptstyle\epsilon_{34}^{5}})$ or $T^4({\scriptstyle\epsilon_{23}^5})$.
\item If $dim\,A=6$, then $A$ has level $5$ if and only if it can be represented by $T^3({\scriptstyle\epsilon_{34}^{6}})$, $T^3({\scriptstyle\epsilon_{45}^{6}})$, $T^{3,2}({\scriptstyle\epsilon_{23}^{6}})$ or $T^4$.
\item If $dim\,A=7$, then $A$ has level $5$ if and only if it can be represented by $T^{2,2}({\scriptstyle\epsilon_{45}^{7}})$, $T^3({\scriptstyle\epsilon_{34}^{7}})$, $T^{2,2,2}({\scriptstyle\epsilon_{24}^7})$, $T^{2,2,2}({\scriptstyle\epsilon_{23}^4-\epsilon_{26}^7+\epsilon_{35}^7})$, $T^{3,2}({\scriptstyle\epsilon_{23}^{7}})$, $T^4$ or $T^{3,3}$.
\item If $8\le dim\,A\le 10$, then $A$ has level $5$ if and only if it can be represented by $T^{2,2}({\scriptstyle\epsilon_{45}^{n}})$, $T^3({\scriptstyle\epsilon_{34}^{n}})$, $T^{2,2,2}({\scriptstyle\epsilon_{24}^n})$, $T^{3,2}({\scriptstyle\epsilon_{23}^{n}})$, $T^4$ or $T^{3,2,2}$.
\item If $dim\,A\ge 11$, then $A$ has level $5$ if and only if it can be represented by $\eta_5$, $T^{2,2}({\scriptstyle\epsilon_{45}^{n}})$, $T^3({\scriptstyle\epsilon_{34}^{n}})$, $T^{2,2,2}({\scriptstyle\epsilon_{24}^{n}})$, $T^{3,2}({\scriptstyle\epsilon_{23}^{n}})$, $T^{2,2,2,2,2}$, $T^{3,2,2}$ or $T^4$.
\end{itemize}
\end{theorem}
\begin{proof} The algebra $A$ can have level $5$ only if $lev\big(IW_1^{max}(A)\big)\le 5$. Then $IW_1^{max}(A)$ is isomorphic to one of the structures ${\bf n}_3$, $T^{2,2}$, $T^3$, $T^{2,2,2}$, $T^{3,2}$, $T^{2,2,2,2}$, $T^4$, $T^{3,2,2}$, $T^{2,2,2,2,2}$ or $T^{3,3}$ where the last case can occur only if $dim\,A=7$. In the first case $A$ has level $5$ if and only if it is represented by $\eta_5$. If $IW_1^{max}(A)\cong T^4$, $IW_1^{max}(A)\cong T^{3,2,2}$, $IW_1^{max}(A)\cong T^{2,2,2,2,2}$ or $IW_1^{max}(A)\cong T^{3,3}$, where in the first case $dim\,A\ge 6$ and in last case $dim\,A=7$, then $A$ has level $5$ if and only if $A\cong IW_1^{max}(A)$.
The remaining part of the theorem follows from Lemmas \ref{T22lev}, \ref{T222lev}, \ref{T2222lev}, \ref{T3lev}, \ref{T32lev}, \ref{T3d4} and \ref{T4lev}.
\end{proof}

\begin{coro} An anticommutative Engel algebra $A$ has infinite level $5$ if and only if it is stably isomorphic to $\eta_5$, $T^{2,2}({\scriptstyle\epsilon_{45}^{n}})$, $T^3({\scriptstyle\epsilon_{34}^{n}})$, $T^{2,2,2}({\scriptstyle\epsilon_{24}^{n}})$, $T^{3,2}({\scriptstyle\epsilon_{23}^{n}})$, $T^{2,2,2,2,2}$, $T^{3,2,2}$ or $T^4$.
\end{coro}

\begin{coro}\label{Mal}
Any anticommutative Engel algebra of level not greater than $5$ is a Malcev algebra. Moreover, any such an algebra is a Lie algebra except the algebra $T^3({\scriptstyle\epsilon_{34}^{n}})$ and the seven-dimensional algebra $T^{2,2,2}({\scriptstyle\epsilon_{23}^4-\epsilon_{26}^7+\epsilon_{35}^7})$ in characteristic not equal to $3$.
\end{coro}

\begin{coro}\label{nilp} Any anticommutative Engel algebra of level not greater than five is nilpotent.
\end{coro}

\begin{rema} Since all algebras in our classification are nilpotent Malcev algebras by Corollaries \ref{nilp} and \ref{Mal}, the classification of anticommutative Engel algebras with level not greater than five of dimension six can be obtained from \cite[Theorem 8]{kppv}.
\end{rema}

\bigskip

{\bf Acknowledgements.} The work was supported by the Russian Science Foundation research project number 19-71-10016. The author is a Young Russian Mathematics award winner and would like to thank its sponsors and jury.

\noindent{{\bf Addresses:}
\newline
Yury Volkov \\
Saint-Petersburg State University\\
Universitetskaya nab. 7-9, St. Peterburg, Russia\\
e-mail:  wolf86\_666@list.ru}

\begin{thebibliography}{99}
\bibitem{BGSS}
Bovdi V., Gerasimova T., Salim M., Sergeichuk V.,
Reduction of a pair of skew-symmetric matrices to its canonical form under congruence,
Linear Algebra Appl., 534 (2018), 17--30.

\bibitem{DK}
Dmytryshyn A., K\aa gstr\"om B., Orbit closure hierarchies of skew-symmetric matrix pencils, SIAM J. Matrix Anal. Appl., 35 (2014), 1429--1443.

\bibitem{khud17}
Francese J., Khudoyberdiyev A., Rennier B., Voloshinov A.,
Classification of algebras of level two in the variety of nilpotent algebras and Leibniz algebras,
J. Geom. Phys., 134 (2018), 142--152.

\bibitem{Gant}
Gantmacher F. R., The Theory of Matrices, Vol. 2, AMS Chelsea Publishing, Providence, RI (1998).

\bibitem{gorb91}
 Gorbatsevich V., 
 On contractions and degeneracy of finite-dimensional algebras, Soviet Math. (Iz. VUZ), 35 (1991), 10, 17--24. 
 
\bibitem{gorb93} 
Gorbatsevich V., Anticommutative finite-dimensional algebras of the first three levels of complexity, 
St. Petersburg Math. J., 5 (1994), 505--521.


\bibitem{IW}
In\"on\"u E., Wigner E.P.,
On the contraction of groups and their representations, Proc. Natl. Acad. Sci. USA, 39 (1953), 510--524.

\bibitem{IvaPal}
Ivanova N. M., Pallikaros C. A.,
On degenerations of algebras over an arbitrary field, AGTA, 7 (2019), 39--83.


\bibitem{kppv}
Kaygorodov I.,  Popov Yu., Volkov Yu., 
Degenerations of binary Lie and nilpotent Malcev algebras, 
 Comm. Algebra,  46 (2018), 11, 4929--4941.

\bibitem{kpv17}
Kaygorodov I., Volkov Yu., 
The variety of $2$-dimensional algebras over an algebraically closed field,  Canad. J. Math., 71 (2019), 4, 819--842.

\bibitem{kpv19}
Kaygorodov I., Volkov Yu., 
Complete classification of algebras of level two, Moscow Math. J., 19 (2019), 3, 485--521.

\bibitem{khud15}
Khudoyberdiyev A., 
The classification of algebras of level two, 
J. Geom. Phys., 98 (2015), 13--20.

\bibitem{khud13}
Khudoyberdiyev A., Omirov B., 
The classification of algebras of level one, 
Linear Algebra Appl., 439 (2013), 11, 3460--3463.

\bibitem{KH}

Koreshkov N. A.,  Haritonov D. U.,
About Nilpotency of Engel Algebras, Russ. Math., 45 (2001), 11, 15--18.

\bibitem{Kuz}

Kuz'min E. N.,
On anticommutative algebras satisfying the engel condition, Sib. Math. J., 8 (1967), 5, 779--785.

\bibitem{S}
Scharlau R., Paare alternierender Formen, Math. Z., 147 (1976), 13--19.

\bibitem{S90}
Seeley C., Degenerations of $6$-dimensional nilpotent Lie algebras over $\mathbb{C}$, Comm. Algebra, 18 (1990), 3493--3505.

\bibitem{Wat}
Waterhouse W., Pairs of symmetric bilinear forms in characteristic $2$, Pacific J. Math., 69 (1977), 1, 275--283.

\end{thebibliography}
\end{document}